\documentclass[12pt]{article}
\usepackage{amsmath,amssymb,amsthm}
\usepackage{fouriernc,graphicx}
\usepackage[top=1in,bottom=1in,left=1in,right=1in]{geometry}

\newtheorem{theorem}{Theorem}[section]
\newtheorem{lemma}[theorem]{Lemma}
\newtheorem{prop}[theorem]{Proposition}
\newtheorem{cor}[theorem]{Corollary}
\newtheorem{deF}[theorem]{Definition}

\newtheorem{example}[theorem]{Example}

\newtheorem{notation}[theorem]{Notation}

\renewcommand{\implies}{\Rightarrow}

\newcommand{\zz}{\mathbb{Z}}
\newcommand{\nat}{\mathbb{N}}
 
\newcommand{\real}{\mathbb{R}}

\newcommand{\dis}{\displaystyle}
\newcommand{\ncom}{\newcommand}

\renewcommand{\implies}{\Rightarrow}

\def\abs#1{\left\vert #1 \right\vert}

\def\set#1{\lbrace #1 \rbrace}

\newcommand{\Inv}{{}^{-1}}

\newcommand{\res}{\operatorname{res}}

\newcommand{\al}{\alpha}
\newcommand{\ome}{\omega}
\newcommand{\ep}{\epsilon}
 \ncom{\Del}{\Delta}

\newcommand{\zec}{Zeckendorf}

\newcommand{\fib}{Fibonacci}
\newcommand{\funds}{fundamental sequence}

\newcommand{\HS}[1]{\rule{#1}{0pt}}

\newcommand{\cE}{\mathcal{E}}

\newcommand{\lex}{lexicographical}
\newcommand{\dord}{<_\text{d}}
\newcommand{\dordeq}{\le_\text{d}}
\newcommand{\aord}{<_\text{a}}
\newcommand{\aordeq}{\le_\text{a}}
\newcommand{\cf}{coefficient function}
\newcommand{\UI}{\mathbf{I}}
\newcommand{\OI}{\UI}
\newcommand{\ord}{\mathrm{ord}}

\newcommand{\moD}[1]{\ (\mathrm{mod}\ #1)}

\newcommand{\wt}{\widetilde}
\newcommand{\wh}{\widehat}

\newcommand{\bbeta}{\bar\beta}
\newcommand{\ddelta}{\bar\theta}

\newcommand{\hbeta}{\hat \beta}

\newcommand{\immsucc}{immediate successor}
\newcommand{\immpred}{immediate predecessor}

\newcommand{\rev}{\mathrm{rev}}
\newcommand{\revbar}{\overline{\rev}}
\newcommand{\shf}{\mathrm{shf}}

\newcommand{\phitilde}{\widetilde{\phi}}
\newcommand{\hdelta}{\hat\delta}
\newcommand{\omtilde}{\widetilde{\ome}}
\newcommand{\dstar}{\delta^*}

\newcommand{\tsum}{\textstyle\sum}
\newcommand{\cEsbar}{\overline{\cE^*}}
\newcommand{\Lsbar}{\overline{L^*}}
\newcommand{\Lstar}{L^*}
\newcommand{\tLstar}{\wt L^*}
\newcommand{\bdecom}{block decomposition}
\newcommand{\decom}{decomposition}
\newcommand{\lesb}{Lesbegue}
\newcommand{\supint}{support interval}
\newcommand{\fin}{_\circ}
\newcommand{\inF}{\mathrm{inf}}

\newcommand{\Htilde}{\widetilde{H}}
\newcommand{\Ltilde}{\widetilde{L}}

\newcommand{\aorcoll}{ascendingly-ordered}

\newcommand{\cEtilde}{\widetilde{\cE}}

\newcommand{\eval}{\mathrm{eval}}

\begin{document}

\begin{center}
\bf\Large
Distribution of \zec\ expressions \\
\rm
\normalsize Sungkon Chang
\end{center}

\begin{quote}
{\bf Abstract:}
By \zec's Theorem, every positive integer is uniquely written as a sum of
distinct non-adjacent \fib\ terms.
In this paper, we investigate the asymptotic formula of the number of
binary expansions $<x$ that have no adjacent terms, and generalize the result to the setting of general linear recurrences with non-negative integer coefficients.
\end{quote}

\section{Introduction} \label{sec:introduction}

\zec's Theorem \cite{zec}
states that each positive integer  is expressed uniquely as a sum of distinct non-adjacent terms of the Fibonacci sequence $(1,2,3,5,\dots)$ where
we reset $(F_1,F_2)=(1,2)$, and 
the expression is called {\it the \zec\ expansion of a positive integer}. 
\zec\ expansions share the simplicity of representation with the binary expansion, but also they are quite curious in terms of the arithmetic operations, determining the $0$th digits, the partitions in \fib\ terms,
the minimal summand property of \zec\ expansions, and its converse; see
\cite{chang-2021,miller:minimality,daykin,grabner,kimberling,fenwick,robbins}.

\zec's Theorem consists of three components; the sequence, the condition on expressions, and the set of numbers represented by the sequence.
By varying each component and shifting the focus to a particular component,
we encounter many interesting questions.
For example, we may vary the \fib\ sequence slightly by changing its initial values,
but maintaining the recurrence relation, and
ask ourselves how many non-negative integers $<x$ are sums of distinct non-adjacent terms of the new sequence---this question is investigated in \cite{chang-2018}.

In this paper, we fix the condition on expressions, 
and change the sequence more than slightly.
The first example we considered was the sequence $\set{2^{k-1}}_{k=1}^\infty$ and \zec's expressions for the \fib\ sequence, i.e.,
we asked ourselves how many non-negative integers $<x$ are expressed as a sum of distinct non-adjacent powers of $2$.
For example, the binary expansion
$165=1 + 2^2 + 2^5 + 2^7 $ satisfies the non-adjacency condition
while the binary expansions of $166$ and $167$ do not. 
For generalization, we reformulate the task as follows.
\begin{deF}\rm \label{def:cZ}
Let $\nat_0$ be the set of non-negative integers.
The infinite tuples  $\mu$ in  $\prod_{k=1}^\infty \nat_0$ are called {\it \cf s}, and
let  $\mu_k$   for $k\in \nat$ denote the $k$th entry of $\mu$, i.e.,
$\mu=(\mu_1,\mu_2,\dots)$.

A set $\wt \cE$ of \cf s $\mu$ is said to be {\it for positive integers} if 
only finitely many entries of $\mu$ are positive for each $\mu\in\wt\cE$, and 
a set of \cf s  is also called a {\it collection}.
Given a collection $\wt\cE$ (of \cf s) for positive integers, an increasing   sequence $\set{\wt H_k}_{k=1}^\infty$ of positive integers 
is called  {\it a fundamental sequence of  $\wt \cE$}
if for each $n\in\nat$, there is unique $\mu\in\wt \cE$ such that 
$n=\sum_{k=1}^\infty \mu_k \wt H_k$.

\end{deF}

The following is immediate from Theorem \ref{thm:integers}.
\begin{lemma}\label{lem:binary-FS}
Let $\wt\cE$ be 
the collection of \cf s $\mu$ for positive integers such that 
$\mu_k\le 1$ for all $k\in \nat$, and 
let $\cE$ be 
the subcollection of $\wt \cE$ consisting of $\mu$ such that 
$\mu_k=1$ implies that $\mu_{k+1}=0$ for all $k\in\nat$.
Then,  
$\set{2^{k-1}}_{k=1}^\infty$ and the \fib\ sequence are the only fundamental sequences of $\wt \cE$ and 
$\cE$, respectively. 
\end{lemma}

\begin{deF}
\rm \label{def:Z-binary}
Let  $\cE$ be the collection defined Lemma \ref{lem:binary-FS}.
If $\mu\in \cE$, then the binary expansion $\sum_{k=1}^\infty \mu_k 2^{k-1}$
is said to be {\it \zec}. 
\end{deF}
\noindent
Then, the earlier task is equivalent to finding
an asymptotic formula of the number of positive integers $ <x$
that have \zec\ binary expansions.

The main goal of this paper is to investigate the asymptotic formula of a function 
that counts the number of positive integers $n$ up to $x$ in
the setting
where the collections $\cEtilde$ and $\cE$ are replaced with {\it periodic \zec\ collections}; see
Definition \ref{def:periodic-collection}.
In \cite{chang-2021}, {\it generalized \zec\ expansions} are introduced in terms of a \lex\ order, which further generalized the expansions introduced in \cite{mw}; see Definition \ref{def:zec-N}.
The expansions introduced in \cite{mw} are for general linear recurrences with constant non-negative integer coefficients, and in the viewpoint of \cite{chang-2021}, the expansions are called {\it periodic}.
The main result of this paper is for the generalized \zec\ expansions that are periodic.
However, several results remain valid for non-periodic ones,
and for this reason,
we use the language introduced in \cite{chang-2021} to present
our work.
Appealing to the reader's intuition,
we formulate the first main result below without properly defining terms.
The terms will be properly introduced in later sections, and its technical version is stated in
Theorem \ref{thm:main-theorem} and \ref{thm:main-minimum}.
\begin{theorem}\label{thm:main-introduction}
Let $\wt\cE$ be a periodic
\zec\ collection for positive integers, and let $\set{\wt H_k}_{k=1}^\infty$ be the unique fundamental sequence of $\wt\cE$.
Let $\cE$ be a periodic
\zec\ subcollection of $\wt\cE$.
Let $z(x)$ be the number of non-negative integers $n<x$ such that 
$n=\sum_{k=1}^\infty \mu_k \wt H_k$ for some $\mu\in\cE$.
Then, (1) there are positive real numbers $\gamma<1$, $a$, and $b$ such that
$a< z(x)/x^\gamma <b$ for all sufficiently large $x$; (2)
there are finitely many explicit and computable sequences $\set{x_k}_{k=1}^\infty$
from which
the limsup and liminf of $z(x)/x^\gamma$ can be determined.

\end{theorem}
\noindent%
For the \zec\ binary expansion, 
using Theorem \ref{thm:main-theorem}, we prove
\begin{equation}\label{eq:lim-inf-sup-2}
\lim\sup_x\ \frac{z(x)}{x^\gamma} =\frac{ \phi+2}5 3^\gamma \approx 1.55,\quad
\lim\inf_x\ \frac{z(x)}{x^\gamma} = \frac{3 \phi+1}5\approx 1.17	
\end{equation}
where $\gamma:=\log_2\phi$ and $\phi$ is the golden ratio; see Section \ref{sec:base-N}.
This is interesting since the values still bear the golden ratio, which must have
come from the expressions in $\cE$; see Section \ref{sec:non-base} and Theorem \ref{thm:N-order}
for more examples of explicit calculations of the bounds.

Let us demonstrate the behavior of $z(x)/x^\gamma$ for the \zec\ binary expansions.
Shown in the first figure of Figure \ref{fig:graph-density} is the graph of $z(x)/x^\gamma$.
The fluctuating behavior of the graph suggests that $z(x)$ may not be asymptotic to an \lq\lq elementary\rq\rq\ increasing function.
However, it reveals certain self-similarities as in some dynamic systems, and in particular, the distribution of
the values are far from being random.

Shown in the second figure of Figure \ref{fig:graph-density} is the
frequency chart of
the values of $z(x)/x^\gamma$ for $262144\le x<349525$, which is one of the maximal intervals in the figure on which the values are not (always) decreasing---the binary expansions of these boundary values will be explained later. The values of the ratio range approximately from $1.17$ to $1.55$, which are represented in the horizontal axis, and we partitioned it into $200$ intervals of equal length, in order to count the number of the values that fall into each of the 200 intervals, which is represented vertically.
The third figure is the probability distribution of the values, i.e.,
it is the graph of $\mathrm{Prob}\set{ 262144\le x<349525 : z(x)/x^\gamma\le r}$ as a function of $r$.

\newcommand{\blankfig}{
\begin{figure}[b]
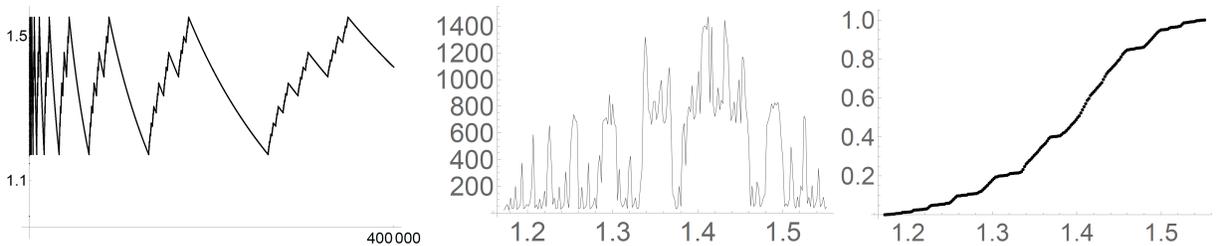
\centering
\includegraphics[height=.14\textheight]{figure.png}
\ %
\includegraphics[height=.14\textheight]{figure.png}
\ %
\includegraphics[height=.14\textheight]{figure.png}
\caption{Graph, frequency, and distribution of $z(x)/x^\gamma$}\label{fig:graph-density}
\end{figure}
}
\newcommand{\realfigpng}{
\begin{figure}[b]\centering
\includegraphics[height=.14\textheight]{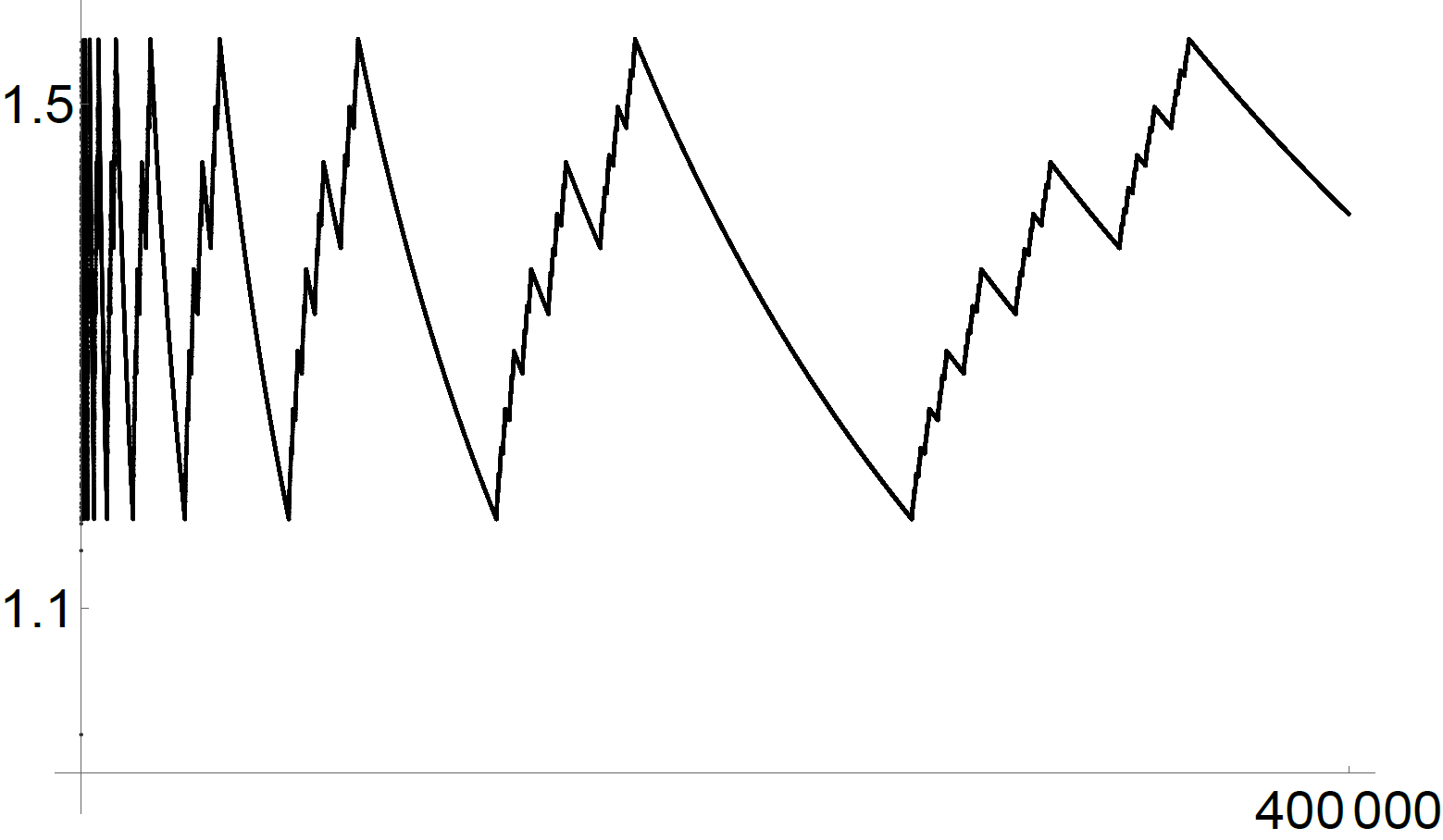} %
\ %
\includegraphics[height=.14\textheight]{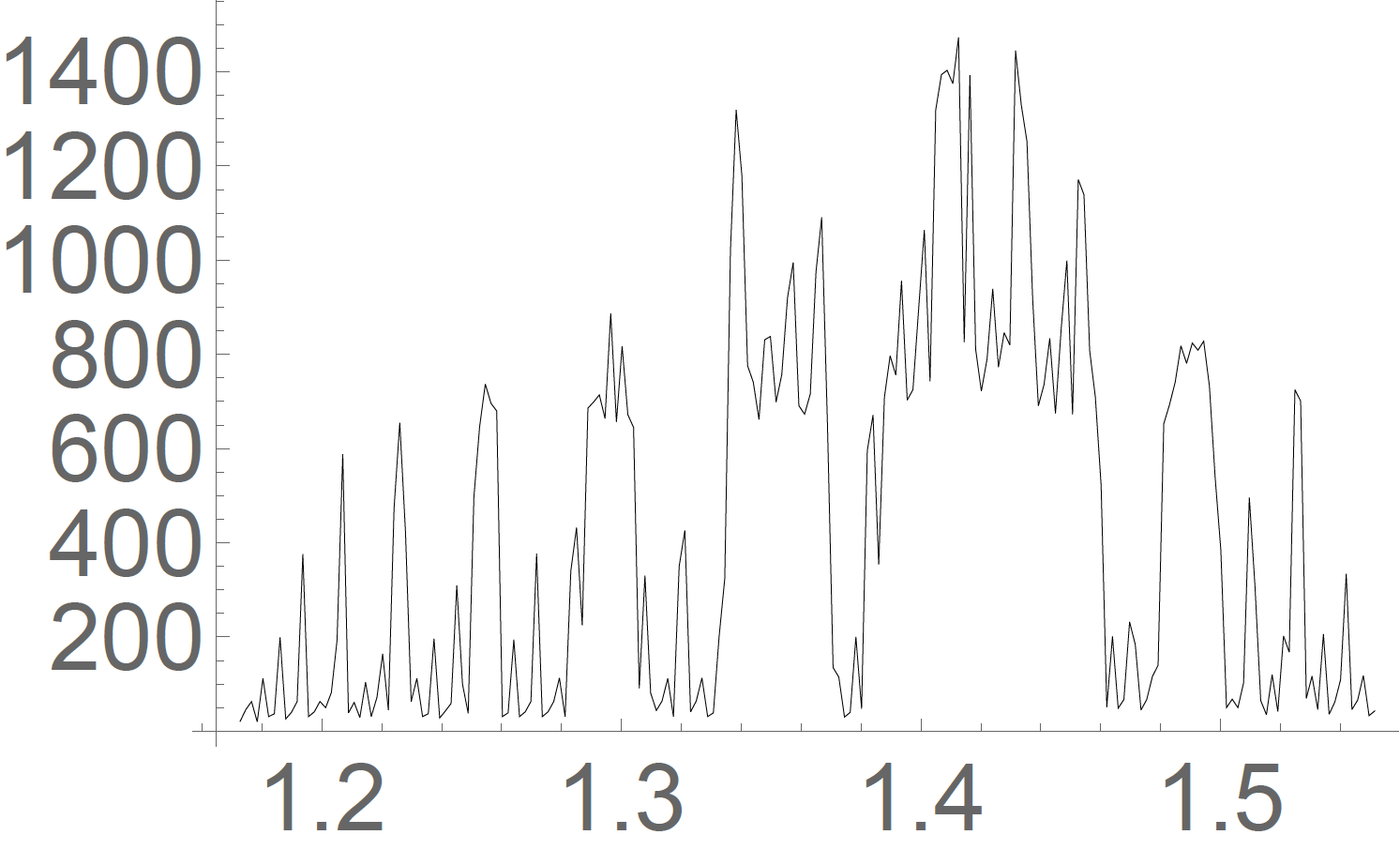}%
\ %
\includegraphics[height=.14\textheight]{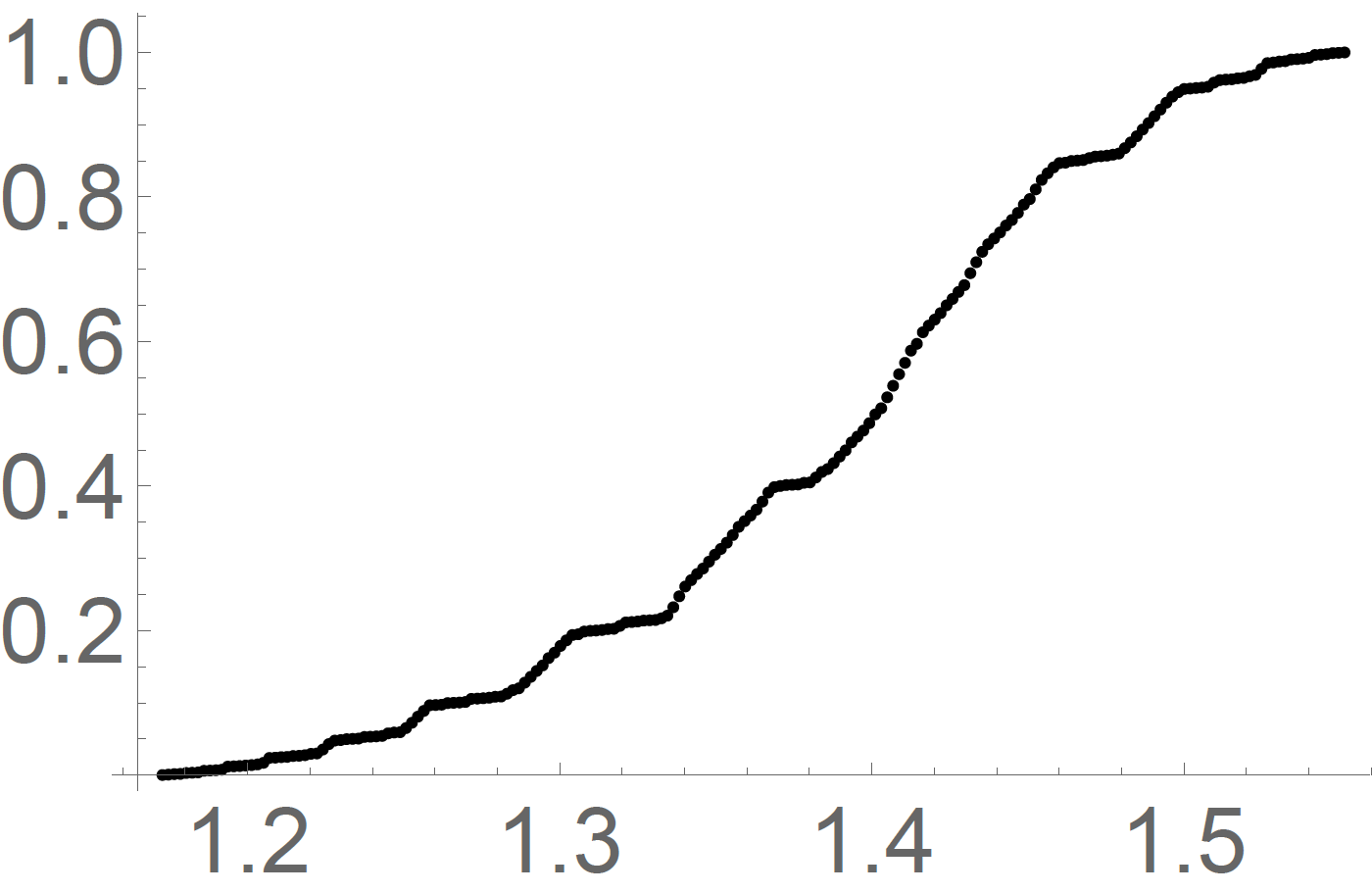}%
\caption{Graph, frequency, and distribution of $z(x)/x^\gamma$}\label{fig:graph-density}
\end{figure}
}
\newcommand{\realfig}{
\begin{figure}[b]\centering
\includegraphics[height=.14\textheight]{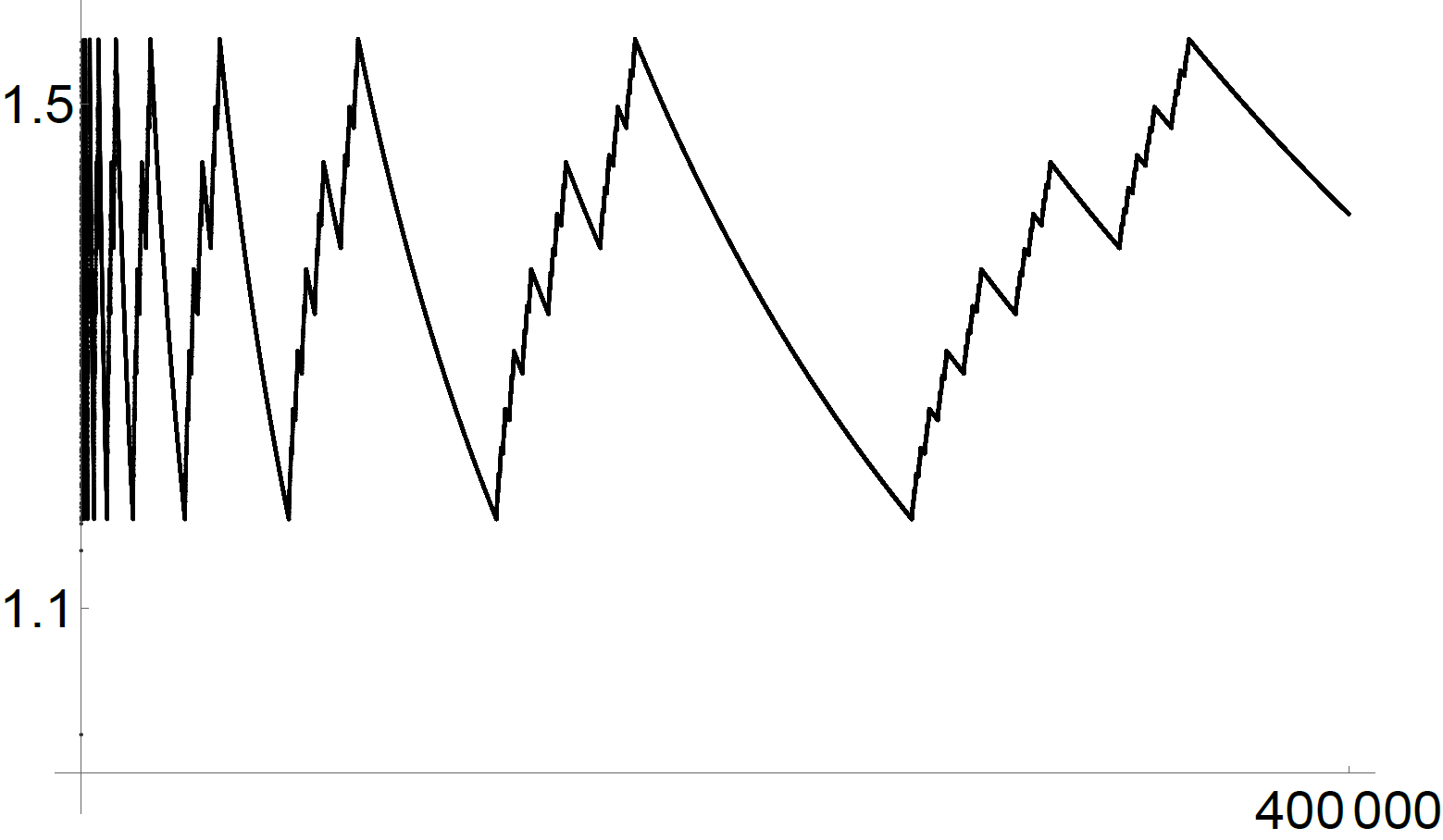} %
\ %
\includegraphics[height=.14\textheight]{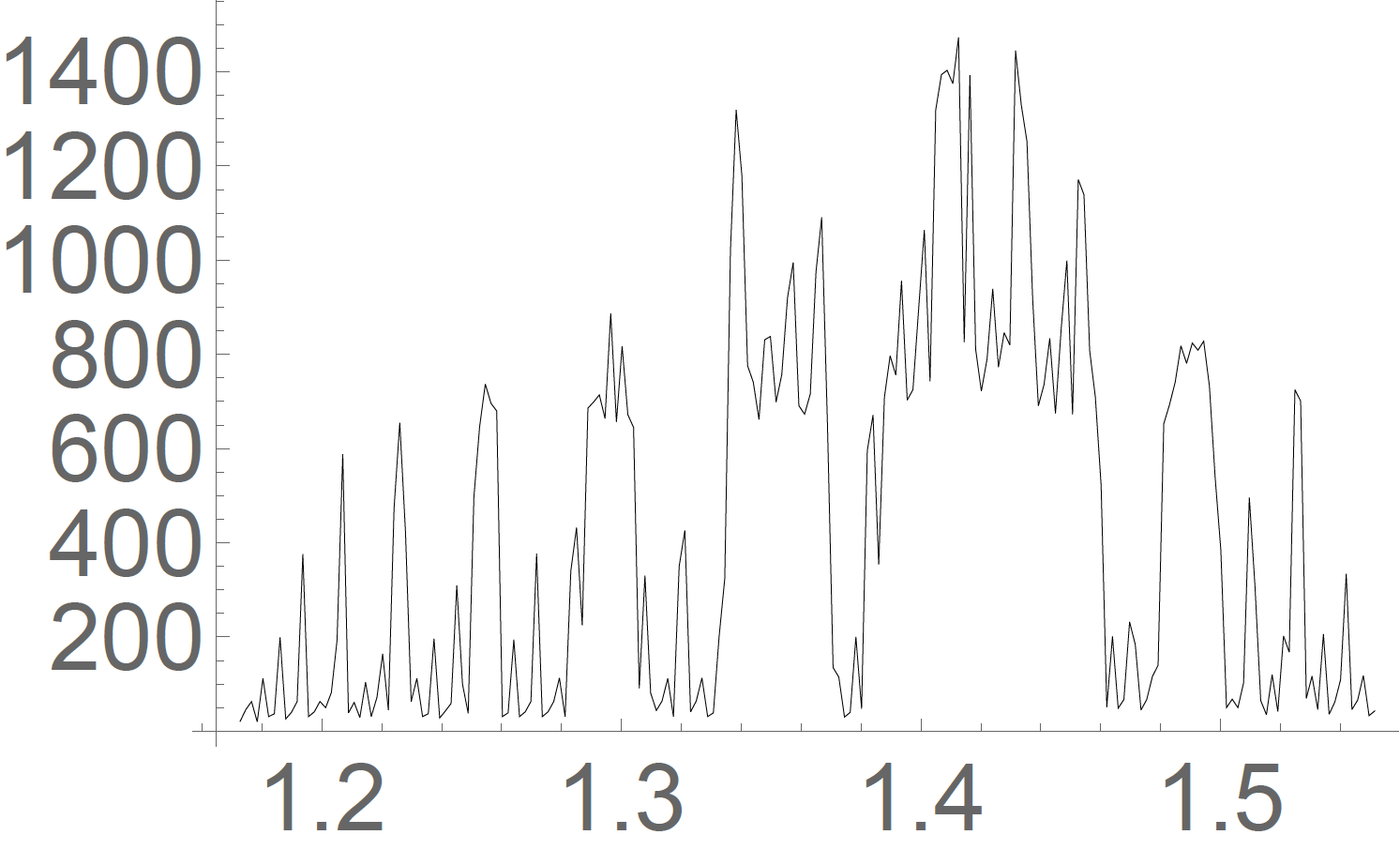}%
\ %
\includegraphics[height=.14\textheight]{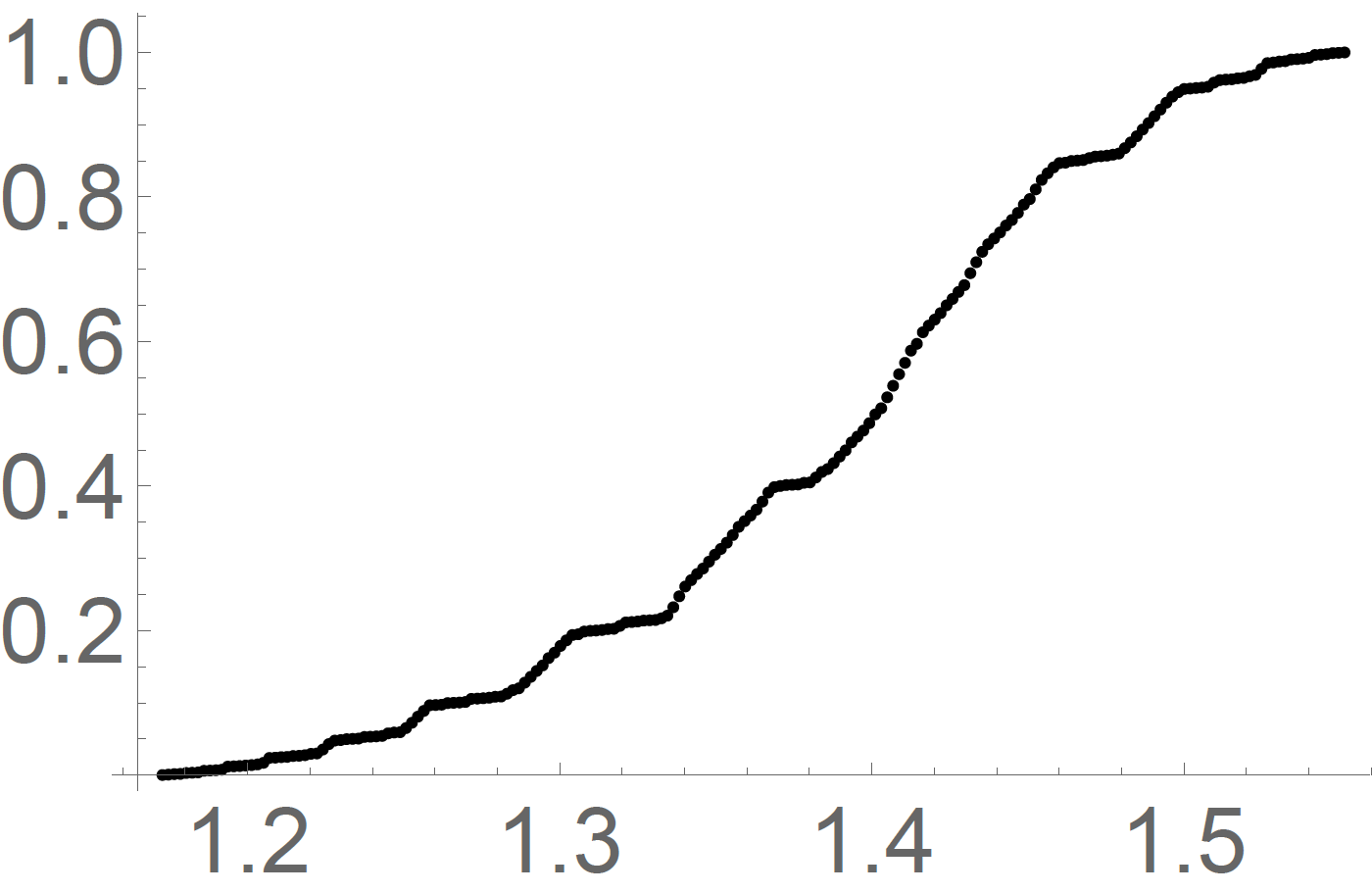}%
\caption{Graph, frequency, and distribution of $z(x)/x^\gamma$}\label{fig:graph-density}
\end{figure}
}
\realfigpng

As observed in the first figure of Figure \ref{fig:graph-density},
there are values of $x$ where the ratios are locally extremal over a relatively large interval, and
we prove that their limits are as identified in (\ref{eq:lim-inf-sup-2}).
The local minima are obtained at $x_1=2^{n-1}$ for each $n\ge 2$, and the local maxima are obtained at
$x_2=\sum_{k=0}^t 2^{n-2k-1}$ where $t=\lfloor (n-1)/2\rfloor$.
The sample interval $262144\le x<349525$ used in the second figure of
Figure \ref{fig:graph-density} is obtained at these values where $n=19$.

The topic has a purely combinatorial interpretation as well.
Consider the sets $\wt\cE$ defined in Lemma \ref{lem:binary-FS}, and 
notice that $\cEtilde$ is \lex ly ordered; see Definition \ref{def:a-order}.
  We may ask ourselves
how many tuples in $\wt\cE$ that are \lq\lq less than\rq\rq\ a certain tuple in $\wt \cE$ and have no consecutive $1$s,
and the fluctuating behavior of the distribution of such tuples in $\wt\cE$
is precisely as presented in Figure \ref{fig:graph-density}.

In \cite{shallit}, the notion of {\it regular sequences} is introduced,
and proved in \cite{heuberger} is an asymptotic formula of $\sum_{k=1}^{x } \chi(k)$ in full generality
where $\chi(k)$ is a regular sequence.
It turns out that the fluctuation behavior is a common feature of the summatory function of a regular sequence, and
there are many examples of summatory functions in the literature that have similar fluctuation behaviors; see \cite{coquet,delange,petho,kirsch}.
This setup in \cite{heuberger} applies to some cases of ours.
For the case of \zec\ binary expansions,
the counting function $z(x)$ is equal to $\sum_{k=0}^{x-1} \chi(k)$ where $\chi$ is the regular sequence
defined by $\chi(k)=1$ if the binary expansion of $k$ is \zec, and $\chi(k)=0$ otherwise.
The asymptotic formula involves a fluctuation factor as a continuous periodic function on the real numbers, and we follow their formulation of the asymptotic formula, which has been standard in the literature;
see Section \ref{sec:gen-reg-seq}.\label{notion:chi}

The authors of \cite{heuberger} describe the fluctuation factor as a Fourier expansion.
In the context of the Mellin-Perron summation formula,
the description of the Fourier coefficients is given in terms of the coefficients of Laurent expansions of the Dirichlet
series $\frac1s\left(\chi(0) + \sum_{k=1}^\infty \chi(k)/k^s\right)$ at special points that line up in a vertical line in the complex plane.
One of our main goals is to obtain exact bounds on the fluctuation factor.
However, it does not seem feasible for us to use the Fourier expansion formulation and obtain exact bounds.
Moreover, the current notion of regular sequences in the literature is formulated for base-$N$ expansions, and
hence, the asymptotic formula in \cite{heuberger} does not apply to
the expansions in terms of  the fundamental sequences  of periodic \zec\ collections.
We shall introduce an example in Section \ref{sec:examples} where we count certain expressions under non-base-$N$ expansions.

Introduced in Theorem \ref{thm:cE-duality} is a formula for the counting function $z(x)$ in full generality,
and we call it {\it the duality formula}.
We use the duality formula to obtain exact bounds on the fluctuation factor of $z(x)$ rather than the
Fourier expansions described in \cite{heuberger}.
The duality formula is a manifestation of the phenomenon that the \cf s of $\cE$ themselves bring out  their own fundamental sequence $\set{H_k}_{k=1}^\infty$
into the formula of $z(x)$.
For example, as stated in Lemma \ref{lem:binary-FS}, the \fib\ sequence
is the fundamental sequence of the \cf s that do not have consecutive $1$s, and the
number of \zec\ binary expansions $<x$ is written
as a sum of \fib\ terms in a fashion similar to the binary expansion of $x$;
see Lemma \ref{thm:2-duality} below.
A subset of $\nat$ is called an index subset in the context of   series expansions,
and an element of the subset is called {\it an index}.
\begin{lemma}[Duality formula]\label{thm:2-duality}
Let $A$ be a finite subset of indices.
If $A$ contains no adjacent indices, define $\bar A:=A$.
If there is a largest index $j\in A$ such that $j+1\in A$, then
define $\bar A:=\set{ k\in A : k\ge j}$.
Then, $z(\sum_{k\in A} 2^{k-1}) = \sum_{k\in \bar A} F_k$.
\end{lemma}
\noindent
For example, $100=2^2+2^5+2^6=\wt H_3+\wt H_6 + \wt H_7$, and
the indices of non-zero coefficients are $A=\set{3,6,7}$.
Thus, $\bar A=\set{6,7}$, and $z(100)= F_6+F_7 = F_8 =34$.

Let us formulate the counting function $z(x)$ using a continuous real-valued function as done in \cite{heuberger}.
Recall the example of the \zec\ binary expansions and the formulas (\ref{eq:lim-inf-sup-2}).
By \cite[Theorem A]{heuberger},
there is a continuous real-valued function $\Phi : \real \to \real$ such that
$z(x) \sim x^\gamma \Phi(\set{\log_2(x)})$ where $\set{\log_2(x)}$ denotes the fractional part of
$\log_2(x) $, and the authors of \cite{heuberger} describe
$\Phi(t)$ as a Fourier series.
As explained earlier, we use the duality formula, rather than the Fourier series description, and we also
use
{\it the generalized
\zec\ expansions for the real numbers in the interval $(0,1)$}, which
is developed in \cite{chang-2021}; see Section \ref{sec:general-real}.
In this section we state our results on the properties of $\Phi$ in Theorem \ref{thm:local-extremum} below, and
its technical versions are found in Theorem \ref{thm:not-diff}, \ref{thm:measure}, and \ref{thm:local-extrm}.
\begin{theorem}\label{thm:local-extremum}
Let $z(x)$ be the counting function defined in Theorem \ref{thm:main-introduction}.
Then, $z(x) \sim x^\gamma \Phi(\set{\log_{\wt\phi}(x)})$ for some positive real numbers $\gamma <1$ and $\wt\phi>1$ and a
continuous function $\Phi$.
Moreover, $\Phi$ is differentiable almost everywhere with respect to Lebesgue measure, and
there are explicit criteria in terms of the generalized \zec\ expansion of a real number $y$ for determining
whether $\Phi$ is differentiable at $y$ or not, and whether $\Phi$ has a local maximum, a local minimum, or neither at $y$.
\end{theorem}

When the quotient $z(x)/x^\gamma$ is calculated,
we noticed that
its values are naturally transitioned to the quotient of two real numbers in the interval $(0,1)$, e.g.,
for the \zec\ binary expansions,
\begin{equation}
\frac{F_9 + F_6 + F_2}{\left(2^8 + 2^5 + 2 \right)^\gamma}
=
\frac{\al \ome+ \al \ome^2 + \al \ome^6 + E}{\left(\frac12 + \frac1{2^2} + \frac1{2^6}\right)^\gamma}
=\al\,\frac{ \ome+ \ome^2 + \ome^6 + E/\al}{\left(\frac12 + \frac1{2^2} + \frac1{2^6}\right)^\gamma}
\label{eq:sample-quotient}
\end{equation}
where $\ome=1/\phi$, $\al=(3\phi+1)/5$, and $E$ is a relatively smaller quantity.
Notice that properly estimating the value of the last expression in (\ref{eq:sample-quotient})
for general cases may rely on
the uniqueness of the expansions of the real numbers in the numerator and the denominator, and it leads us naturally to
the generalized \zec\ expansions of real numbers.
Motivated from the non-constant quotient factor of the last expression in (\ref{eq:sample-quotient}),
we define a function on the interval $(0,1)$ to represent the factor, and denote it by
$\delta^* : (0,1)\to \real$; see Definition \ref{def:rev-bar}.
Theorem \ref{thm:not-diff}, \ref{thm:measure}, and \ref{thm:local-extrm} are stated for $\delta^*$.
For the case of the \zec\ binary expansions, the relationship between the continuous functions
can be $\Phi(\set{\log_2(x)}) \sim \al \delta^*(2^{\set{\log_2{x}}-1})$ where
$\Phi$ is the function appearing in \cite{heuberger}.

As demonstrated in the case of \zec\ binary expansions,
the factor $\delta^*$ is infinitesimally fluctuating for the general cases as well, and
the main tools for analyzing the fluctuations were
the generalized \zec\ expansions for the real numbers (see the work toward Definition \ref{def:rev-bar}) and
Lemma \ref{lem:gamma-ineq}, which describes a sufficient condition for $\delta^*$
having a higher value.

The remainder of the paper is organized as follows.
In Section \ref{sec:general-zec}, we review the generalized \zec\ expansions for the positive integers and the real numbers in $(0,1)$.
These contents are also available in \cite{chang-2021}, but
for the readability of our work
we review the contents in this paper.
In Section \ref{sec:zec-exressions}, the setup of two generalized \zec\ expressions is introduced for the positive integers and for the interval $(0,1)$, and in this setup
the duality formula and the transition from the integers to the real numbers are introduced.
In Section \ref{sec:proofs}, the main results Theorem \ref{thm:main-introduction} and \ref{thm:local-extremum}
are proved.
In Section \ref{sec:examples},
calculations for the \zec\ binary expansions and two more pairs of collections are demonstrated, one of which
is a pair with non-base-$N$ expansions.
In Section \ref{sec:gen-reg-seq}, we conclude the paper discussing the generalization of regular sequences and their summatory functions.

\section{Generalized \zec\ expansions }\label{sec:general-zec}

We shall review the definitions and results related
to the generalized \zec\ expansions for positive integers that are periodic, which are also available in \cite{chang-2021}, and
by \cite[Theorem 7]{chang-2021},
our definition is equivalent to the definition introduced in \cite[Definition 1.1]{mw}.
We also review the definitions and results related
to the generalized \zec\ expansions for the real numbers in the interval $\OI:=(0,1)$ that are periodic;
see \cite{chang-2021} for non-periodic ones   for $\OI$.

\subsection{Notation and definitions}
We identify a sequence of numbers with an infinite tuple, and
denote its terms using subscripts.
For example, the sequence of positive odd integers is denoted by $Q=(1,3,5,7,\dots)$, and we use
subscripts to denote its values, e.g., $Q_3 = 5$.
We may define
a sequence by describing $Q_k$ where $k$ is assumed to be an index
$\ge 1$, e.g., the earlier example is the sequence
given by $Q_k = 2k-1$. 
Recall \cf s from Definition \ref{def:cZ}.
Given a \cf\ $\ep $ and a sequence $Q$,
we denote 
$\sum \ep Q$ to be the formal sum $\sum_{k=1}^\infty \ep_k Q_k$.

Given two \cf s $\ep$ and $\mu$, we define $\ep+\mu$ to be the \cf\ such that
$(\ep+\mu)_k=\ep_k + \mu_k$ for $k\ge 1$, and $\ep-\mu$ to be the \cf\ such that
$(\ep-\mu)_k=\ep_k - \mu_k$ for $k\ge 1$ if $\ep_k\ge \mu_k$ for all $k\ge 1$.
Given $c\in\nat_0$, we define $c\ep:=(c\ep_1,c\ep_2,\dots)$.

\begin{deF}\rm \label{def:beta-i}
Let $\beta^i$ be the \cf\ such that $\beta^i_k=0$ for all $k\ne i$ and
$\beta_i^i=1$, and call it {\it the $i$th basis \cf}; in particular, $\beta^i_k$ does not denote
the $i$th power of an integer.
\end{deF} 

\begin{deF}\label{def:bar}\rm
For \cf s $\ep$, we use the bar notation $\bar a$ to represent the repeating entries.
For example, $\ep=(1,2,3,\bar 0)$ is
the \cf\ such that $\ep_k=0$ for all $k>3$, and we denote it by $\ep=(1,2,3)$ as well.
If $\ep=(\bar 0, 1,2,3)$,
then it means that there is an index $n$ such that $\ep_k=0$ for $1\le k\le n$, and $\ep_{n+j}=j$ for $1\le j\le 3$.
We use the simpler notation $0$ for the zero \cf\ $(\bar 0)$.
If there is an index $M$ such that $\ep_k=0$ for all $k >M$, $\ep$ is said to
{\it have finite support}.
\end{deF}

\begin{deF}\label{def:restriction}
\rm  
Let $\ep$ be a \cf, and let $m\le n$ be two positive integers.
Define
$\rev_n(\ep)$ to be the \cf\ $(\ep_n,\ep_{n-1},\dots, \ep_1)$.

Define $\res_m(\ep)$ and $\res^m(\ep)$ to be the \cf s such that
$\res_m(\ep)_k=0$ for all $k> m$,
$\res_m(\ep)_k=\ep_k$ for all $1\le k\le m$,
$\res^m(\ep)_k=\ep_k$ for all $k\ge m$, and
$\res^m(\ep)_k=0$ for all $1\le k< m$.
Also we define $\res^0(\ep):=\ep$, $\res_0(\ep)=0$, and $\res_n^m(\ep)$ to be the \cf\ such that
$\res_n^m(\ep)_k = \ep_k$ for all $m\le k\le n$
and $\res_n^m(\ep)_k = 0$ for other indices $k$. 
\end{deF}

\subsection{Generalized \zec\ expansions for positive integers }\label{sec:zec-integers}

We shall use {\it the ascending \lex\ order} to define generalized \zec\ expansions for positive integers, and
it is defined as follows.

\begin{deF}
\rm \label{def:a-order}
Given two \cf s $\mu$ and $\mu'$ with finite support, if there is a largest positive integer $k$ such that
$\mu_j = \mu'_j$ for all $j>k$ and
$\mu_k < \mu'_k$, then we denote the property by $\mu \aord \mu'$.
\end{deF}
\noindent
For example, if $\mu=(1,2,10,3,7)$ and $\mu'=(1,3,1,4,7)$,
then $\mu \aord \mu'$ since $\mu_4 < \mu'_4$ and $\mu_5=\mu'_5$.

\begin{deF}\label{def:collection}\rm
We define
{\it
an \aorcoll\ collection $\cE$ of \cf s} to be
a set of \cf s with finite support ordered by the ascending \lex\ order that contains the zero \cf\ and all basis \cf s $\beta^i$.

Let $\cE$ be an \aorcoll\ collection, and let $\mu\in\cE$.
The smallest \cf\ in $\cE$ that is greater than $\mu$,
if (uniquely) exists, is called {\it the \immsucc\ of $\mu$ in $\cE$}, and we denote it by $\wt\mu$.
The largest \cf\ in $\cE$ that is less than $\mu$,
if (uniquely) exists, is called {\it the \immpred\ of $\mu$ in $\cE$}, and we denote it by $\wh\mu$.
In particular, $\hbeta^n$ shall denote
the \immpred\ of the $n$th basis \cf\ $\beta^n$ for each $n\ge 2$, if (uniquely) exists.
\end{deF}

\begin{deF}\label{def:order}\rm
Let $\mu$ be an element of an \aorcoll\ collection.
If $\mu \ne 0$,
the largest index $n$ such that $\mu_n \ne 0$ is called the order of $\mu$,
denoted by $\ord(\mu)$. If $\mu=0$, we define
$\ord(\mu)=0$.
\end{deF}
\noindent
By definition, an \aorcoll\ collection $\cE$ contains all the basis \cf s, i.e., $\beta^{n-1}\in\cE$ for $n\ge 2$, and hence, the \immpred\ $\hbeta^n$, if exists,
has
a non-zero value at index $n-1$, i.e., $\ord(\hbeta^n)=n-1$.

\begin{deF}\label{def:zec-N}\rm
Let $\cE$ be an \aorcoll\ collection of \cf s.
The collection is called {\it\zec} if it satisfies the following:
\begin{enumerate}
\item For each $\mu\in\cE$ there are at most finitely many \cf s that are less than $\mu$.

\item
Given $\mu\in \cE$, if
its \immsucc\ $\wt\mu$ is not $ \beta^1+\mu $,
then there is an index $n\ge 2$
such that
$\res_{n-1}(\mu)=\hbeta^n$ and $\wt\mu = \beta^n
+ \res^n(\mu)$.

\end{enumerate}
We also call $\cE$ {\it a \zec\ collection for positive integers}.
\end{deF}
\noindent
Surprisingly enough, it turns out that a \zec\ collection for positive integers is completely determined by the subset $\set{\hbeta^n : n\ge 2}$, which is the meaning of Theorem \ref{cor:generate-zec} below.

\begin{theorem}
[\cite{chang-2021}, Definition 8 \&\ Corollary 9]\label{cor:generate-zec}
Given \cf s $\theta^n$ of order $n-1$ for $n\ge 2$,
there is a unique \zec\ collection for positive integers such that $\hbeta^n=\theta^n$ for each $n\ge 2$.
\end{theorem}
Let us demonstrate the \zec\ collection for positive integers determined by $\theta^n$, which have a common periodic structure.
\begin{deF}\rm \label{def:periodic}
Let $L=(e_1,\dots,e_N)$ be a list of non-negative integers where $N\ge 2$ and $e_1 \ne 0$, and
let $\beta^*$ be the \cf\ such that $\beta^*_k = e_j$ where $1\le j\le N$ and $k\equiv j\moD N$.
Given $n\ge 2$, define $\theta^n:=\rev_{n-1}(\beta^*)$.

For an integer $n\ge 2$,
a \cf\ $\zeta$ is called {\it a proper $L$-block at index $n-1$} if
there is an index $k\le n-1$ such that $\res^{k+1}(\zeta)=\res^{k+1}(\theta^n)$, $\zeta_k<\theta^n_k$,
and $\res_{k-1}(\zeta)=0$.
The index interval $[k,n-1]$ is called {\it the support interval} of the proper $L$-block $\zeta$ at index $n-1$.
The \cf s $\theta^n$ for $n\ge 2$ are called {\it maximal $L$-blocks}, and the index interval $[1,n-1]$
is called {\it the support interval} of the maximal $L$-block at index $n-1$.
Both proper and maximal $L$-blocks are called {\it $L$-blocks}, and
two $L$-blocks $\zeta$ and $\xi$ are said to be {\it disjoint} if their support intervals are disjoint sets of indices.
\end{deF}

\begin{theorem}[\cite{chang-2021}, Theorem 7]\label{thm:periodic-zec}
Let $L$ and $\theta^n$ for integers $n\ge 2$ be as defined in Definition  \ref{def:periodic}.
Then, the collection $\cE$  consisting of the sums of finitely many disjoint
$L$-blocks is a \zec\ collection for positive integers such that $\hbeta^n=\theta^n$ for $n\ge 2$.
\end{theorem}

\begin{deF}\rm \label{def:periodic-collection}
Let $L$ and $\cE$ be as defined in Theorem \ref{thm:periodic-zec}.
The collection $\cE$  is called
{\it the periodic \zec\ collection for positive integers determined by $L$}.
\end{deF}

\begin{deF}\rm \label{def:L-decompositions}
Let $L$ be the list defined in Definition \ref{def:periodic}.
Let $\ep = \sum_{m=1}^n \zeta^m$ where $\zeta^m$ are disjoint $L$-blocks at index $i_m$ such that
$i_m < i_{m+1}$ for $1\le m <n$.
This expression is called {\it an $L$-block decomposition}, and
if $\ep\ne 0$, the support intervals of $\zeta^m$ form a partition of $[1,i_n]$, and $\zeta^n\ne 0$, 
the summation is called {\it the (full) $L$-block decomposition of $\ep$}.
If $\ep\ne 0$ and $\zeta^m$ are non-zero disjoint $L$-blocks, then
the summation is called {\it the non-zero $L$-\bdecom of $\ep$}.
The expression $\mu + \tau$ is called an {\it $L$-\decom}\ %
if $ \sum_{m=1}^n \zeta^m$ is {\it an $L$-\bdecom}, $\mu=\sum_{m=1}^s \zeta^m$, and
$\tau=\sum_{m=t}^n \zeta^m$ where $1\le s<t \le n$.
\end{deF}

\begin{example}\rm \label{exm:periodic}
Let $L=(2,3,0)$, and let $\cE$ be the periodic \zec\ collection for positive integers determined by $L$.
Then, $\beta^*=(2,3,0,2,3,0,2,3,0,\dots)$, and the following are examples of the \immpred s:
$$
\hbeta^2=(2),\ \hbeta^3=(3,2),\ \hbeta^4=(0,3,2),\
\hbeta^5=(2,0,3,2),\ \hbeta^9=(3,2,0,3,2,0,3,2).$$
Listed below are examples of proper $L$-blocks at index $5$:
$$
\zeta^1 = (0,0,0,0,2),\ \zeta^2=(0,0,0,2,2),\ \zeta^3=(0,0,0,3,2).$$
The support intervals of $\zeta^k$ for $k=1,2,3$ are $[4,5]$, $[4,5]$, and $[2,5]$, respectively.
Stated below are two examples of the sum of finitely many disjoint $L$-blocks where a semicolon is inserted instead of a comma to indicate the end of the support interval of a block, and their \immsucc s are stated below as well:
\begin{gather*}
\ep=(2,0,3,2; 0;0, 0,3,2;2,2;1),\ \tau=(1,0,3,2; 0;0, 0,3,2;2,2;1)\\
\wt\ep=(0;0;0;0;1;0, 0,3,2;2,2;1),\ \wt\tau=(2,0,3,2; 0;0, 0,3,2;2,2;1).
\end{gather*}
\end{example}

Given a collection $\cE$ of \cf s with finite support and a sequence $G$ of positive integers,
let $\eval_G : \cE \to \nat_0$ be the evaluation map given by $\ep\mapsto \sum \ep G$. The weak converse of \zec's theorem for \zec\ collections for positive integers is
stated below.  Recall Definition \ref{def:a-order}.
\begin{theorem}[\cite{chang-2021}, Theorem 16, Lemma 3\ \&\ 37]
\label{thm:integers}
Let $\cE$ be a \zec\ collection for positive integers with the \immpred s $\hbeta^n$ for \
$n\ge 2$. Then, there is a unique increasing sequence $H$ of positive integers such that
$\eval_H : \cE \to \nat_0$ is bijective.
Moreover, for each $n\ge 2$,
\begin{equation}\label{eq:full-recursion}
H_n = H_1 + \sum\hbeta^n H,
\end{equation}
and the map $\eval_H$ is increasing, i.e.,
$\ep \aord \delta$ if and only if $
\eval_H(\ep)< \eval_H(\delta)$.
\end{theorem}
\noindent
Recall Definition \ref{def:cZ}.
Then, by Theorem \ref{thm:integers}, 
given  a \zec\ collection $\cE$ for positive integers
there is one and only one
  fundamental sequence of $\cE$.
  
  \begin{deF}\label{def:periodic-Z-expansions}\rm
Given a \zec\ collection $\cE$ for positive integers, the expansion $n=\sum \ep H$
is called the $\cE$-expansion of $n$ where $H$ is the fundamental sequence of $\cE$.
\end{deF}
  
For the collection in Example \ref{exm:periodic}, the fundamental sequence $H$ is given by
the linear recurrence $H_{n+3} = 2 H_{n+2} + 3 H_{n+1} + H_n$ with
$(H_1,H_2,H_3)=(1,3,10)$.
In general, the \funds\ of a periodic \zec\ collection determined by $L=(e_1,\dots,e_N)$ is given by
\begin{equation}\label{eq:recurrence}
H_n=\sum_{k=1}^{N-1} e_k H_{n-k} + (1+e_N) H_{n-N},\quad
n>N
\end{equation}
with initial values $(H_1,H_2,\dots,H_N)$ determined by (\ref{eq:full-recursion}).
\begin{deF}
\rm \label{def:char-poly}
Let $L=(e_1,\dots,e_N)$ be the list defined in Definition \ref{def:periodic}.
The following is called {\it the characteristic polynomial of $L$ for positive integers:}
\begin{equation}\label{eq:char-polynomial}
x^N - \sum_{k=1}^{N-1} e_k x^{N-k} - (1+e_N ).
\end{equation}
\end{deF}

By Descartes' Rule of Signs, the   polynomial defined in (\ref{eq:char-polynomial}) has one positive simple real root $\phi$, and it is $>1$.
By \cite[Section 5.3.2]{chang-2021}, it is the only root with the largest modulus in
the complex plane.
Thus, by Binet's formula, $H_n =\al\phi^n + O(\phi^{nr})$ where $\al$ and $r<1$ are positive real constants.
Notice that the periodic \zec\ collection for positive integers determined by $L= (e_1,\dots,e_N,e_1,\dots, e_N)$
is equal to the collection determined by $ (e_1,\dots, e_N)$
since they have the same \immpred\ of $\beta^n$ for each $n\ge 2$.

\subsection{Generalized \zec\ expansions for real numbers } \label{sec:general-real}

Let us review the definitions and results for the periodic \zec\ collections for the real numbers in the interval
$\OI:=(0,1)$.
The general definition is given in \cite[Definition 10]{chang-2021}, and
in this paper,
we review the definition for the periodic \zec\ collections.

\begin{deF}
\rm \label{def:d-order}
Given two \cf s $\ep$ and $\ep'$, we define
{\it the descending \lex\ order} as follows.
If there is a smallest positive integer $k$ such that
$\ep_j = \ep'_j$ for all $j<k$ and
$\ep_k< \ep'_k$, then we denote the property by $\ep \dord \ep'$.
\end{deF}
\noindent
For example, if $\ep=(1,2,10,5,\dots)$ and $\ep'=(1,3,1,10,\dots)$,
then $\ep \dord \ep'$ since $\ep_2 < \ep'_2$ and $\ep_1=\ep'_1$.
\begin{deF}
\rm \label{def:ord-star}
If $\ep \ne 0$ is a \cf,
the smallest index $n$ such that $\ep_n \ne 0$ is
denoted by $\ord^*(\ep)$. If $\ep=0$, we define
$\ord^*(\ep)=\infty$.
\end{deF}

\begin{deF}\rm \label{def:periodic-real}
Let $L=(e_1,\dots,e_N)$ be a list of non-negative integers where $N\ge 2$ and $e_1 \ne 0$.
Given $n\ge 1$, let $\bbeta^n$ be the \cf\ such that
$\res_{n-1}(\bbeta^n)=0$ and $\bbeta^n_k = e_j$ for $k\ge n$ where $1\le j\le N$ and $k-n+1\equiv j\moD N$,
and we call $\bbeta^n$ {\it the maximal $L^*$-block at index $n$}.

For an integer $n\ge 2$,
a \cf\ $\zeta$ is called {\it a proper $L^*$-block at index $n$} if
there is an index $k\ge n$ such that $\res_{k-1}(\zeta)=\res_{k-1}(\bbeta^n)$, $\zeta_k<\bbeta^n_k$,
and $\res^{k+1}(\zeta)=0$.
The index interval $[n,k]$ is called {\it the support interval} of the proper $L^*$-block $\zeta$ at index $n$.
Two proper $L^*$-blocks $\zeta$ and $\xi$ are said to be {\it disjoint} if their support intervals are disjoint sets
of indices.

The descendingly ordered  collection
of non-zero \cf s consisting of the finite or infinite sums of disjoint
proper $L^*$-blocks is called {\it the periodic \zec\ collection for $\OI$ determined by $L$}. 
If $\cE$ denotes the periodic \zec\ collection  for positive integers determined by $L$,
the collection for $\OI$ is denoted by $\cE^*$.

Let $\ep = \sum_{m=1}^\infty \zeta^m$ where $\zeta^m$ are disjoint proper $L^*$-blocks at index $i_m$ such that
$i_m < i_{m+1}$.
This expression is called {\it an $L^*$-block decomposition of $\ep$}, and
if the support intervals of $\zeta^m$ form a partition of $[1,\infty)$,
the summation is called {\it the (full) $L^*$-block decomposition of $\ep$}.
If  $\zeta^m$ are non-zero disjoint proper $L^*$-blocks, then
the summation is called {\it the non-zero $L^*$-block decomposition of $\ep$}.
The expression $\mu + \tau$ is called {\it an $\Lstar$-\decom} if
$\sum_{m=1}^\infty \zeta^m$ is an $\Lstar$-\bdecom,
$\mu=\sum_{m=1}^s \zeta^m$, and $\tau = \sum_{m=t }^\infty \zeta^m$
where $1\le s < t$.

\end{deF}

Thus, given a list $L=(e_1,\dots,e_N)$, we have the periodic \zec\ collection $\cE$ for positive integers  and 
the periodic \zec\ collection $\cE^*$ for $\OI$.
The collection $\cE$ is ascendingly ordered, and $\cE^*$ is descendingly ordered.
Also notice that the zero \cf\ and the maximal blocks $\bbeta^n$ for $n\ge 1$ are not members of $\cE^*$,
but the zero \cf\ can be a proper $\Lstar$-block at any index $n$.

\begin{example}\rm \label{exm:periodic-real}
Let $L=(2,3,0)$, and let $\cE^*$ be the periodic \zec\ collection for $\OI$ determined by $L$.
The following are examples of maximal $L^*$-blocks:
$$
\bbeta^1=(2,3,0,2,3,0,\dots),\ \bbeta^2=(0,2,3,0,2,3,0,\dots),\ \bbeta^3=(0,0,2,3,0,2,3,0,\dots) .$$
Listed below are examples of proper $L^*$-blocks at index $4$, and semicolons are
inserted to indicate the ends of the support intervals.  See Definition \ref{def:bar} for the bar notation:
$$
\zeta^1 = (0;0;0;2,0;\bar 0),\ \zeta^2=(0;0;0;2,2;\bar 0),\ \zeta^3=(0;0;0;2,3,0,0;\bar 0).$$
The support intervals of $\zeta^k$ for $k=1,2,3$ are $[4,5]$, $[4,5]$, and $[4,7]$, respectively.
Stated below are two examples of \cf s where $\ep\not\in\cE^*$ and $\tau\in\cE^*$:
\begin{gather*}
\ep=(1,2,2,\bar 0) + \bbeta^4 = (1,2,2,\overline{2,3,0}),\quad
\quad\tau= (1,2,2, \overline{2,3,0,0}).
\end{gather*}
\end{example}

\newcommand{\temp}{\cite[Definition 10, Theorem 13]{chang-2021}}
\begin{lemma}[\cite{chang-2021}, Definition 10, Theorem 13] \label{lem:rev-ep}
Let $\ep$ be a \cf.
Then, $\rev_n(\ep)\in \cE$ for all $n\ge 1$ if and only if $\ep\in \cE^*$ or
$\ep=\sum_{m=1}^n \zeta^m + \bbeta^{b+1}$ where
$\sum_{m=1}^n \zeta^m$ is an $L^*$-block decomposition and $[i,b]$ is the support interval of $\zeta^n$.
\end{lemma}

Recall Definition \ref{def:d-order}.
\begin{theorem}[\cite{chang-2021}, Theorem 24]
\label{thm:real}
Let $\cE^*$ be the periodic \zec\ collection for $\OI$ determined by $L=(e_1,\dots,e_N)$. Then, there is a unique decreasing sequence $Q$ of real numbers in $\OI$ such that the map
$\eval_Q : \cE^* \to \OI$ given by $\ep \mapsto \sum \ep Q$ is bijective,
and the map $\eval_Q$ is increasing, i.e.,
$\ep \dord \delta$ if and only if $
\eval_Q(\ep)< \eval_Q(\delta)$.

Moreover, for each $n\ge 1$,
\begin{equation}\label{eq:full-recursion-real}
Q_n = \sum\bbeta^{n+1} Q,\ \text{and}\ Q_n = \ome^n
\end{equation}
where $\ome$ is the (only) positive real zero of the polynomial
\begin{equation}\label{eq:char-polynomial-real}
-1+\sum_{k=1}^{N-1} e_k x^{k} + (1+e_N)x^N.
\end{equation}
\end{theorem}
\begin{deF}
\rm \label{def:FS-real}
Let $\cE^*$ be a periodic \zec\ collection for $\OI$.
The sequence $Q$ defined in Theorem \ref{thm:real} is  
called {\it the fundamental sequence of $\cE^*$}, and 
the polynomial (\ref{eq:char-polynomial-real}) is called 
 {\it the characteristic polynomial of $L$ for $\OI$.}
\end{deF} 

For the collection in Example \ref{exm:periodic-real}, the fundamental sequence $Q$ is given by
the linear recurrence $Q_{n } = 2 Q_{n+1} + 3 Q_{n+2} + Q_{n+3}$, and their initial values
are given by the formula in (\ref{eq:full-recursion-real}).
In general, the \funds\ of the periodic \zec\ collection for $\OI$ determined by $L =(e_1,\dots,e_N)$ is given by
\begin{equation}\label{eq:recurrence-real}
Q_n=\sum_{k=1}^{N-1} e_k Q_{n+k} + (1+e_N) Q_{n+N},\quad
n\ge 1 .
\end{equation} 
The polynomial in (\ref{eq:char-polynomial-real}) is a reciprocal version of the polynomial $f(x)$ in (\ref{eq:char-polynomial}), i.e., it is equal to
$-x^N f(1/x)$.
Thus, it has a positive simple real zero $\ome<1$, and it is the only zero with the smallest modulus.

For the remainder of Section \ref{sec:general-real}, let $\cE^*$ be the periodic \zec\ collection for $\OI$ determined by $L$, and
let $Q$ denote the \funds\ of $\cE^*$.
The results introduced below will be used in Section \ref{sec:proofs}, but 
they are introduced here in order to provide the reader with an opportunity to 
become more familiar with generalized \zec\ expansions of real numbers.
\begin{deF}\label{def:inf-version}\rm
If $\ep\in \cE^*$, we define $\inF_{\cE^*}(\ep)$ as follows.
If $\ep$ does not have finite support, then $\inF_{\cE^*}(\ep):=\ep$.
If $c$ is the largest index such that $\ep_c\ge 1$,
then $\inF_{\cE^*}(\ep)=\res_c(\ep)-\beta^c +\bbeta^{c+1}$,
which is not a member of $\cE^*$.
If there is no confusion, let $\inf$ denote $\inf_{\cE^*}$.
\end{deF}
\noindent

The long recursion (\ref{eq:full-recursion-real}) implies the following lemma, and we leave the proof to the
reader.
\begin{lemma}
If $\ep\in\cE^*$, then $\sum\ep Q = \sum\inf (\ep) Q$.
\end{lemma}

Recall the restriction notation from Definition \ref{def:restriction} and 
$\Lstar$-decompositions from Definition \ref{def:periodic-real}.
Throughout the proof of Lemma \ref{lem:ep-converge}, we use the increasing property of 
$\eval_Q$ defined in Theorem \ref{thm:real}.
\begin{lemma}\label{lem:ep-converge}
Let
$\ep\in\cE^*$,
let
$x=\sum\ep Q >0$,
and let $\Delta x$ be a sufficiently small positive real number such that $x+\Delta x<1$
and $x-\Delta x>0$.
If $x+\Delta x=\sum\ep^+ Q$ for $\ep^+\in \cE^*$, and
$n$ is the largest integer $\ge 0$ such that
$\res_n(\ep)=\res_n(\ep^+)$, then $n\to\infty$ as $\Delta x\to 0$.
If $x-\Delta x=\sum\ep^- Q$ for $\ep^-\in \cE^*$, and
$n$ is the largest integer $\ge 0$ such that
$\res_n(\inF(\ep))=\res_n(\ep^-)$, then $n\to\infty$ as $\Delta x\to 0$.

\end{lemma}

\begin{proof}
Let $ \ep=\sum_{m=1}^\infty \zeta^m $ be an
$\Lstar$-\bdecom, and let $[i_m,b_m]$ be the \supint\ of
$\zeta^m$.
Notice that given $m\ge 1$, there is a smallest index $k_m>b_m$
such that $\bbeta^{i_m}_{k_m}>0$, and that $b_m$ and $ k_m$ approach $\infty$
as $m\to\infty$.

Let $\Delta x=\sum\ep^0 Q$ where $\ord^*(\ep^0)=t$, and
let $T$ be the largest index such that $k_T< t$ where $t$ is sufficiently large.
Then, $\Delta x < Q_{k_T}$.
Recall from Definition \ref{def:periodic-real} the support intervals of proper $\Lstar$-blocks.
Then, $\xi:=\zeta^T + \beta^{b_T}$ is a proper $\Lstar$-block, and its \supint\ is either
$[i_T,b_T]$ or $[i_T,k_T]$,
and hence, $ \sum_{m=1}^{T-1 } \zeta^m + \xi + \beta^{k_T}\in \cE^*$.
Thus,
\begin{gather*}
x=\sum\left(\tsum_{m=1}^{\infty } \zeta^m\right) \dis Q<
x+\Delta x<
\dis\sum\left(\tsum_{m=1}^{T } \zeta^m\right) Q + Q_{b_T} + \Delta x
<\dis\sum\left(\tsum_{m=1}^{T-1 } \zeta^m\right) Q +\dis\sum\xi Q
+ Q_{k_T} .
\end{gather*}
By Theorem \ref{thm:real},
$\tsum_{m=1}^{\infty } \zeta^m\dord \ep^+ \dord\tsum_{m=1}^{T-1 } \zeta^m + \xi+\beta^{k_T}$,
and by the definition of the \lex\ order, we have
$\res_{b_{T-1}}(\ep^+)=\tsum_{m=1}^{T-1 } \zeta^m=\res_{b_{T-1}}(\ep )$.
Notice that as $\Delta x\to 0$, we have $t\to\infty$ and $T\to\infty$,
and hence, $b_{T-1}\to\infty$.

Let us prove the statement about $x-\Delta x$.
Suppose that $ \ep=\sum_{m=1}^\infty \zeta^m $ is the non-zero $\Lstar$-\bdecom.
Then, $\inF(\ep)=\ep$, and
\begin{gather*}
x=\sum\left(\tsum_{m=1}^{\infty } \zeta^m\right) \dis Q>
x-\Delta x >
\dis\sum\left(\tsum_{m=1}^{T -1} \zeta^m\right) Q +\sum\zeta^T Q - \Delta x
>\dis\sum\left(\tsum_{m=1}^{T -1} \zeta^m\right) Q .
\end{gather*}
As in the earlier case,
we have $\res_{b_{T-1}}(\ep^-)=\tsum_{m=1}^{T-1 } \zeta^m=\res_{b_{T-1}}(\ep )$,
and $b_{T-1}\to\infty$ as $t\to\infty$.

Suppose that $ \ep=\sum_{m=1}^\ell \zeta^m $ is the non-zero $\Lstar$-\bdecom,
and let $c$ be the largest index such that $\zeta^\ell_c\ge 1$.
Then, $\inF(\ep)=\sum_{m=1}^{\ell-1}\zeta^m + \xi + \bbeta^{c+1}$
where $\xi=\zeta^\ell-\beta^{c}$.
Then, there is a largest index $p$ such that $p<t$
and $c +1\equiv p\moD N$ where $N$ is the length of $ L$, and hence,
by the periodic structure of the entries of $\bbeta^{c+1}$, 
we have $\bbeta^{c+1} = \res_{p-1}(\bbeta^{c +1})+\bbeta^p$.
Thus, 
\begin{gather*}
x=\sum\inF(\ep) Q
>
x-\Delta x =
\dis\sum\left(\tsum_{m=1}^{\ell -1} \zeta^m\right) Q +\sum\xi Q
+\sum\res_{p-1}(\bbeta^{c +1}) Q +\sum\bbeta^p Q- \Delta x \\
>\dis\sum\left(\tsum_{m=1}^{\ell -1} \zeta^m\right) Q +\sum\xi Q
+\sum\res_{p-1}(\bbeta^{c +1}) Q
=\sum\res_{p-1}(\inF(\ep)) Q.
\end{gather*}
Thus, $\sum( \tsum_{m=1}^{\ell -1} \zeta^m +\zeta^\ell)Q >
\sum\ep^- Q >\sum\res_{p-1}(\inF(\ep)) Q$.
It's clear that $\res_{c-1}(\ep)=\res_{c-1}(\ep^-)=\res_{c-1}(\inF(\ep))$
and $\ep^-_{c}=\zeta^\ell_{c}-1=\inF(\ep)_c$.
Thus, $\ep^-=\res_c(\ep^-) + \res^{c+1}(\ep^-)$ is an $\Lstar$-decomposition, and
hence,
$ \res^{c+1}(\ep^-)$
is a member of $\cE^*$.
By Theorem \ref{thm:real},
$\sum\ep^- Q >\sum\res_{p-1}(\inF(\ep)) Q$
implies $\res_{p-1}(\inF(\ep))\dord \ep^-$, and hence,
$\res^{c+1} _{p-1}(\bbeta^{c+1})=\res^{c+1} _{p-1}(\inF(\ep)) \dord \res^{c+1}(\ep^-) $.
Thus, we have
$\res^{c+1} _{p-1}(\bbeta^{c+1}) \dord \res^{c+1}(\ep^-) \dord \bbeta^{c+1}$,
and hence, $\res_{p-1}(\ep^-)=\res_{p-1}(\inF(\ep))$.
Since $p\to\infty$ as $t\to\infty$, we prove the result.
\end{proof}

\subsection*{Acknowledgments}
 The author would like to thank the referees who brought attention to the works that are related to the author's work.

\section{Generalized \zec\ expressions}\label{sec:zec-exressions}

\subsection{Expressions for positive integers}\label{sec:zec-integers-expressions}
Throughout this section,
let $\cE$ and $\cEtilde$ be periodic \zec\ collections for positive integers determined by lists 
$L=(e_1,\dots,e_N)$ and
$\wt L =(\wt e_1,\dots,\wt e_M)$, respectively,
such that $\cE$ is a proper subcollection of $\cEtilde$.
Let $H$ and $\Htilde$ be their fundamental sequences,
respectively. Let $\hbeta^n$ and $\hat\theta^n$ be the \immpred\ of $\beta^n$ in $\cE$ and $\cEtilde$, respectively.
Recall from Definition \ref{def:periodic-Z-expansions} that  $ \sum\ep \Htilde$ 
is called  an $\cEtilde$-expansion 
if $\ep\in \cEtilde$,
and if, in addition, $\ep\in\cE$, then the summation, which is written in terms of $\wt H$, is called {\it an $\cE$-expression}.
The main object of this paper is the function that counts 
the number of positive integers $<x$ whose $\cEtilde$-expansion  is an $\cE$-expression.

The proofs of Proposition \ref{thm:subcollection}, Lemma \ref{thm:reduction}, and Theorem \ref{thm:cE-duality} given below in this section remain valid for non-periodic \zec\ collections.
The main result of this paper is for periodic ones, and they are stated for periodic collections.

\begin{prop}\label{thm:subcollection}
Let $\cEtilde$ and $\cE$ be  \zec\ collections for positive integers.
Then, the  collection $\cE$ is a subcollection of $\cEtilde$ if and only if
the \immpred s $\hbeta^n$ in $\cE$ are members of $\cEtilde$ for $n\ge 2$.
\end{prop}

\begin{proof}
If $\cE$ is a subcollection of $\cEtilde$, then $\hbeta^n\in \cE$ for $n\ge 2$ are members of $\cEtilde$. Suppose that $\hbeta^n$ for $n\ge 2$ are members of $\cEtilde$, and let us show that $\cE\subset \cEtilde$.
First let us show that
the proper $L$-blocks are members of $\cEtilde$.
Let $n\ge 2$, and let $\hbeta^n =\sum_{m=1}^M \zeta^m$ be the $\wt L$-block decomposition (including zero blocks);
see Definition \ref{def:periodic}.
Let $\xi$ be a non-zero proper $L$-block with support interval $[a,n-1]$, i.e., $\xi_a<\hbeta^n_a$ and
$\xi_k = \hbeta^n_k$ for all $a<k\le n-1$, and
let us show that $\xi\in\cEtilde$.

Notice that $a$ is contained in the support interval of $\zeta^{m_0}$ for some $1\le m_0\le M$, and 
let $[b,s]$ be the support interval where $b\le a\le s$.
Then, $\zeta^{m_0}$ is a nonzero $\wt L$-block since $\xi_a<\hbeta^n_a=\zeta^{m_0}_a$ implies that $\zeta^{m_0}_a>0$. Let $\eta$ be the \cf\ such that $\eta_a=\xi_a$,
$\eta_k=\zeta^{m_0}_k$ for all $a< k \le s=\ord(\zeta^{m_0})$, and
$\eta_k=0$ for other indices $k$.
Then, $\eta$ is a proper $\wt L$-block with support interval $[a,s]$
since $\eta_a=\xi_a<\hbeta^n_a=\zeta^{m_0}_a\le \hat\theta^{s+1}_a$ and
$\eta_k = \zeta^{m_0}_k=\hat\theta^{s+1}_k$ for $a<k\le s$.
Notice that $\xi=\eta+ \sum_{m=m_0+1}^M \zeta^m$, which is the sum of disjoint $\wt L$-blocks,
i.e., $\xi\in \cEtilde$.
Moreover, the support interval of the proper $L$-block $\xi$ is $[a,n-1]$, and
the (disjoint) union of the support intervals of the $\wt L$-blocks $\eta$ and $\zeta^m$ for $m_0+1\le m \le M$
is $[a,n-1]$ as well. Thus, in general, a disjoint sum of proper $L$-blocks is a member of $\cEtilde$.

Let us consider the case where the sum involves a maximal $L$-block $\hbeta^c$.
Let $\ep$ be the sum of  disjoint proper $L$-blocks $\xi^m$ for $1\le m \le T$, i.e.,
 $\ep = \sum_{m=1}^T \xi^m$, and suppose that 
  $[c,n]$ where $c>1$ is the union of the support intervals of the proper $L$-blocks $\xi^m$.
 Then, as shown earlier,   the union of the support intervals of the  $\wt L$-block decompositions of $\xi^m$
 for $1\le m\le T$ is 
$J_1:=[c,n]$ as well.
Recall that $\hbeta^c\in \cEtilde$, and hence, it is the  $\wt L$-block decomposition, the union of whose support intervals
is $J_2:=[1,c-1]$.  Since $J_1$ and $J_2$ are disjoint, the $L$-decomposition $\hbeta^c + \ep$ is 
an $\wt L$-decomposition, and hence, it is a member of $\cEtilde$ as well.
This
proves that all $L$-\bdecom s are members of $\cEtilde$.

\end{proof}

\begin{example}\rm
Let us demonstrate examples of subcollections.
For each of the following, we first check if
$\hbeta^n \aordeq \hat\theta^n$ for $n=2,\dots,7$,
and check if $\hbeta^n$ is a member of $\cEtilde$.
We use the semicolon to indicate the end of the \supint\ of an $\wt L$-block.
\begin{enumerate}
\item
Let $\wt L=(1,3)$ and $L=(1,2,1)$.
Then, $\hbeta^7=(1,2;1,1;2,1) \aord \hat\theta^7=(3,1,3,1,3,1)$.
However, $\hbeta^7$ does not have an $\wt L$-\bdecom\ since
$\res^3(\hbeta^7)$ has an $\wt L$-\bdecom, but $\res_2(\hbeta^7)$ does not an $\wt L$-\bdecom.
Thus, $\cE$ determined by $L$ is not a subcollection of $\cEtilde$.
\item
Let $\wt L=(3,2)$ and $L=(2,3,1)$.
Then, $\hbeta^7=(;1,3;2;1,3;2) \aord \hat\theta^7=(2,3,2,3,2,3)$, and
$\hbeta^7$ has an $\wt L$-\bdecom.
This example is sufficient to understand that $\hbeta^n$ is a member of $\cEtilde$ for all $n\ge 2$.
\item
Let $\wt L=(3,2)$ and $L=(3,1,2)$.
Then, $\hbeta^7=(;2;1,3;2;1,3) \aord \hdelta^7=(2,3,2,3,2,3)$, and
$\hbeta^7$ has an $\wt L$-\bdecom.
This example is sufficient to understand that $\hbeta^n$ is a member of $\cEtilde$ for all $n\ge 2$.

\end{enumerate}
\end{example}

Let us compare the \funds s of $\cE$ and $\cEtilde$ below.
We identify  later in Theorem \ref{thm:alpha} the values of $\al$ and $\wt \al$ that are mentioned in Theorem \ref{thm:ratios}.
\begin{theorem}\label{thm:ratios}
There are positive real numbers
$\al$, $\widetilde{\al}$, $\phi$, $\phitilde$, and $r<1$ such that
\begin{gather}\label{eq:three-ratios}
H_n = \al \phi^{n-1} + O(\phi^{n r}),\quad
\Htilde_n = \widetilde{\al}\phitilde^{n-1}+ O(\phitilde^{n r}),\quad
\text{and}\quad
\lim_{n\to\infty} \frac{H_n}{\Htilde_n^\gamma}
=\frac{\al}{\widetilde{\al}^\gamma}>0
\end{gather}
where $\gamma=\log_{\phitilde}\phi$ and $1<\phi<\phitilde$.
\end{theorem}

\begin{proof}
Recall $L=(e_1,\dots,e_N)$ and $\wt L=(\wt e_1,\dots,\wt e_M)$.
Let $\mathbf b=(e_N,\dots,e_1)$ and $\mathbf{\widetilde b}=
(\tilde e_M,\dots, \tilde e_1)$, and let $f$ and $\tilde f$ be the characteristic polynomials of $L$ and $\wt L$, respectively.
As explained in the paragraphs below (\ref{eq:char-polynomial}),
there are positive real numbers
$\al$, $\widetilde{\al}$, $\phi$, $\phitilde$, and $r$ such that
(\ref{eq:three-ratios}) holds and $\phi>1$,
where $\phi$ and $\wt \phi$ are the dominant positive real zeros of $f$ and $\wt f$, respectively.

Let us prove that $\phitilde>\phi $. Since $\cE\ne \cEtilde$,
if we duplicate the repeating blocks of $\cE$ and
$\cEtilde$ to the length of $NM$ as follows, Proposition \ref{thm:subcollection} implies
\begin{equation*}
\mathbf b'=(e_{N },\dots,e_N,\dots,e_1,e_N,\dots,e_1)\aord\mathbf{\widetilde b}'=(\tilde e_{ M},\dots,
\tilde e_M,\dots,
\tilde e_1,\tilde e_M,\dots,\tilde e_1).
\end{equation*}
Notice that
$\cE$ is equal to the periodic \zec\ collection determined by the reverse of the list $\mathbf b'$, which is equal to $\hbeta^{1+NM }$, and
$\cEtilde$ is equal to the periodic \zec\ collection determined by the reverse of the list $\mathbf{\widetilde b}'$, which is equal to $\hat\theta^{1+NM }$.
By Proposition \ref{thm:subcollection}, there is a largest index $s\le NM$ such that
$\hbeta^{1+NM}_s < \hat\theta^{1+NM }_s$.
Let $f_*(x):=x^{ NM}-(e_1 x^{ NM-1}+\cdots + e_{N-1} x +(1+e_{N}))$
be the characteristic polynomial of $\rev_{NM}(\mathbf b')$ for positive integers.
Then, the following induction step shows that $\phi$ is a zero of $f_*$:
\begin{align*}
\phi^{kN}&=e_1 \phi^{ kN-1}+\cdots + e_{N-1} \phi +(1+e_{ N }) \\
\implies
\phi^{kN+N}&=e_1 \phi^{ k N-1 +N}+\cdots + e_{ N-1} \phi^{1+N} +(1+e_{N })\phi^N \\
\implies
\phi^{(k+1)N} &
=e_1 \phi^{ (k+1)N-1 }+\cdots + e_{ N-1} \phi^{1+N}
+ e_{ N }\phi^N+ e_1\phi^{N-1}+\cdots +(1+e_N).
\end{align*}
Since $\phi$ is the only positive real root of $f_*$, it would be sufficient
to prove that $f_*(\phitilde)>0$.

Notice that for any integer $t>3$,
$\hbeta^{1+tNM }_{s+(t-1)NM}< \hat\theta^{1+tNM}_{s+(t-1)NM}$.
For convenience, let $\hbeta:=\hbeta^{1+tNM}$,
$\hat\theta :=\hat\theta^{1+tNM}$,
$s_*:=s+(t-1)NM$, and $s_1:=1+(t-1)NM$.
Let $y:=\sum\res^{s_1}(\hbeta)\Htilde$.
Below we shall establish that $\Htilde_{1+tNM} - (y+\Htilde_{s_1})$
is greater than a term $\Htilde_{t_*}$ where $t_*<1+tNM$ is sufficiently close to $1+tNM$.
Then, the asymptotic version of the inequality will imply that $f_*(\phi)$ is positive.

Notice below that if $s=1$, i.e., $s_*=1+ (t-1)NM=s_1$, then
$\sum\res_{s_*-1}^{s_1}(\hbeta )\Htilde$ is interpreted as $0$;
\begin{align*}
\Htilde_{(t-2)NM} 	 + \Htilde_{s_1} + y
& =\Htilde_{(t-2)NM} 	+ \Htilde_{s_1}
+\sum\res_{s_*-1}^{s_1}(\hbeta )\Htilde
+
\sum\res^{s_*}(\hbeta )\Htilde
\\
& = \Htilde_{(t-2)NM} 	+ \Htilde_{s_1}
+\sum\res_{s_*-1}(\hbeta)\wt H - \sum \hbeta^{s_1} \wt H
+
\sum\res^{s_*}(\hbeta )\Htilde.\\
\intertext{Notice that
by the definition of $s$, the \cf\ $\res^{s_*}(\hbeta)$ is a proper $\Ltilde$-block,
which implies that  $\res_{s_*-1}(\hbeta) \in\cEtilde$.  Thus, }
\Htilde_{(t-2)NM} 	 + \Htilde_{s_1} + y
& < \Htilde_{(t-2)NM} 	+ \Htilde_{s_1}
+\Htilde_{s_*}-\sum \hbeta^{s_1}\wt H
+
\sum\res^{s_*}(\hbeta )\Htilde
\\
&\le
\Htilde_{(t-2)NM}-\sum \hbeta^{s_1}\wt H 	+ \Htilde_{s_1}+
\sum \res^{s_*}(\hat\theta ) \Htilde.
\end{align*}
Notice that
$0<y_1:=\Htilde_{(t-2)NM}-\sum \hbeta^{s_1}\wt H 	+ \Htilde_{s_1}<\Htilde_{s_1}$
implies there is $\ep^0\in\cEtilde$ such that $\ord(\ep^0)\le s_1-1$
and $y_1=\sum\ep^0\wt H$.
Let $\ep^0=\sum_{m=1}^\ell \zeta^m$ be an $\wt L$-\bdecom\ such that
the \supint\ of $\zeta^\ell$ is $[a,s_1-1]$, and let us claim that $\ep:=\ep^0 + \res^{s_*}(\hat\theta )\in \cEtilde$.
Let $[s_0,(t-1)NM]$ be the \supint\ of the proper $\wt L$-block
$\res^{s_*}(\hat\theta )$.
Notice that the \supint\ of the $\wt L$-block $\res^{s_*}(\hat\theta )$ may not be
$[s_*,(t-1)NM]$ or 
$[s_*-1,(t-1)NM]$, e.g.,
if $\wt L=(1,3,0)$ and $s_*=2+(t-1)NM$, then the \supint\ is $[s_1-1,(t-1)NM]$.
Since $\wt e_1=\hat\theta_{s_1-1}>0$ and $\res^{s_*}(\hat\theta )_{s_1-1}=0$, we have
$s_1-1\le s_0$.
If $s_1\le s_0$,
then $\ep$ is a sum of two disjoint $\wt L$-blocks.
If $s_1-1=s_0$, then
$\hat\theta_{s_1}=\wt e_M$ implies that
$\zeta^\ell + \res^{s_1}(\hat\theta )$ forms a single $\wt L$-block.
Thus, $\ep^0+\res^{s_*}(\hat\theta )\in \cEtilde$, and 
\begin{align*}
\Htilde_{(t-2)NM} 	 + \Htilde_{s_1} +y
& \le \sum\ep^0 \wt H+
\sum \res^{s_*}(\hat\theta ) \Htilde \le
\sum \hat\theta \Htilde < \Htilde_{1+tNM} \\
\implies \quad \Htilde_{(t-2)NM}
&< \Htilde_{tNM+1} - (y + \Htilde_{s_1})
\\
\implies
\quad
\phitilde^{-NM-1} &\le \phitilde^{ NM}
-(e_1 \phitilde^{ NM-1}+\cdots +(1+e_{N}))+ o_t(1).
\end{align*}
This proves that
$0<\phitilde^{-NM-1}\le f_*(\phitilde) $ as $t\to\infty$.
Since $f_*(\phi)=0$, we prove that $\phitilde>\phi$.
\end{proof}

\subsection{The duality formula}
Recall Lemma \ref{thm:2-duality}, which is the specialized version of the duality formula.
We introduce the general version for \zec\ collections for positive integers, including non-periodic ones.

\begin{deF}\label{def:cE}
Given a non-zero \cf\
$\mu$ with finite support that is not necessarily in $\cEtilde$, we denote by $\check\mu$ the largest \cf\ in $\cE$ that is $\aord \mu$, and we denote the \immsucc\ of $\check\mu$ in $\cE$ by $\bar\mu$.
\end{deF}

\begin{lemma}[Mixed decomposition]\label{lem:ep-bar}
Given $\ep\in \cEtilde-\cE$, write $\ep =\zeta^0 +\mu$ as an $\wt L$-decomposition such that
$\mu=\sum_{m= 1}^n \zeta^m$ is an $ L$-block decomposition with the maximal integer $n\ge 0$.
If $n\ge 1$, then let $[a,i]$ be the support interval of $\zeta^1$, and 
if $n=0$, then let $a=\ord(\zeta^0)+1$.
Then, $\hbeta^a\aord\zeta^0 $ and $\bar\ep = \beta^a +\sum_{m=1}^n \zeta^m$.
\end{lemma}

\begin{proof}
Notice that $\ep \not \in\cE$ implies that $a\ge 2$.
Suppose that $\zeta^0 \aord \hbeta^{a}$.
There is an index $j\le a-1$ such that $\zeta^0_j< \hbeta^{a}_j$ and
$\zeta^0_k= \hbeta^{a}_k$ for $j+1\le k\le a-1$.
Thus, $\res^{j}_{a-1}(\zeta^0)$ is a proper $L$-block
with \supint\ $[j,a-1]$
, and this contradicts the maximality of $n$.
Notice that $\zeta^0=\hbeta^{a}$ also contradicts the maximality of $n$.
Thus, $\hbeta^{a}\aord \zeta^0$, and $\hbeta^{a}+\sum_{m=1}^n \zeta^m
\aord \zeta^0+\sum_{m=1}^n \zeta^m$.
Notice that $\check\ep = \hbeta^{a}+\sum_{m=1}^n \zeta^m$, and hence,
$\bar\ep = \beta^a +\sum_{m=1}^n \zeta^m$.
\end{proof}

\begin{example}\label{exm:ep-bar}\rm
Let $L=(1,0)$ and $\wt L=(1,1)$.
Suppose that $\ep=(1,1,1,0,1,0,1)$, which is a member of $\cEtilde - \cE$.
Notice that there is no $L$-block at index $7$ in $\ep$, and hence, the maximal number $n$ described in Lemma \ref{lem:ep-bar} is $0$, and
$a=8$.
So, $\hbeta^8\aord \ep$, and hence, $\bar\ep = \beta^8$.
\end{example}
\begin{deF}\rm \label{def:counting-function}
Given a positive integer $x$, let $R(x)$ denote
the subset of non-negative integers $\sum\ep \Htilde<x$ where $\ep\in\cE$, and
let $z(x)$ denote $\#R(x)$.
\end{deF}\noindent

\begin{lemma}\label{thm:reduction}
For each $\mu \in \cEtilde$,
$R(\sum\mu \Htilde) = R(\sum\bar\mu\Htilde)$.
\end{lemma}

\begin{proof}
Let $y\in R(\sum\bar\mu\Htilde)$.
Then $y=\sum\ep\Htilde$ where $\ep\in \cE$ and
$\sum\ep\Htilde<\sum\bar\mu\Htilde$.
By Theorem \ref{thm:integers} applied to $\cEtilde$ and the fundamental sequence $\Htilde$,
we have $\ep\aord \bar \mu $, and
hence,
$\ep \aordeq \check\mu \aord\bar \mu$.
Since $\check\mu\aord\mu$, we have $\ep \aord \mu$.
Thus,
$y=\sum\ep\Htilde < \sum \mu \Htilde$, and hence,
$y\in R(\sum \mu \Htilde)$.

Let $y=\sum\ep\Htilde \in R(\sum \mu\Htilde)$ for some $\ep\in\cE$.
Then, by Theorem \ref{thm:integers}, $\ep\aord\mu$, and hence, $\ep\aordeq \check\mu<\bar\mu$.
Thus, $y=\sum\ep \Htilde \le \sum\check\mu\Htilde<\sum\bar\mu\Htilde$, i.e.,
$y\in R(\sum\bar\mu\Htilde)$.

\end{proof}

\begin{theorem}[Duality Formula]\label{thm:cE-duality}
For each $\mu \in \cEtilde$, we have
$z\left( \sum\mu\Htilde \right) =\sum\bar\mu H$.
\end{theorem}

\begin{proof}
By Lemma \ref{thm:reduction},
$z(\sum\mu \Htilde)=z(\sum\bar\mu \Htilde)$.
If Theorem \ref{thm:integers} is applied
to $\cE$ and $H$, then
$\eval_H$ restricts to a bijective
function
from $S:=\set{ \ep \in \cE : \ep \aord \bar \mu}$
to $T:=\set{n\in \nat_0 : n<\sum\bar\mu H}$.
Theorem \ref{thm:integers} can be
applied to $\cEtilde$ and $\Htilde$ as well.
Since
$\cE$ is a subcollection of $\cEtilde$,
$\eval_{\wt H}$ restricts to a bijective
function
from $S:=\set{ \ep \in \cE : \ep \aord \bar \mu}$
to $\wt T:=R(\sum\bar\mu \Htilde)$. 
In other words, the three subsets $S$, $T$, and $\wt T$ are bijective to each other.
Thus, $z(\sum\bar\mu \Htilde)=\#\wt T =\#S=\#T = \sum \bar\mu H$,
which proves that $z(\sum\mu \Htilde)=z(\sum\bar\mu\Htilde)
=\sum \bar\mu H$.

\end{proof}

Recall the setup in Example \ref{exm:ep-bar}.
Then, $H_k=F_k$ is the \funds\ of $\cE$ determined by $L$.
By the duality formula, $z(\sum\ep \wt H)=\sum \bar\ep F=F_8$.
However, it turns out that
using $\mu=(0,1,1,0,1,0,1)$ is also convenient for this case.
Notice that the linear recurrence of $F$ implies that
$\sum\mu F = \sum\bar\ep F$.
This observation is formulated in Lemma \ref{thm:2-duality}.

\subsection{Transition to expressions for the interval $\OI$}
By the limit value in (\ref{eq:three-ratios}) and the Duality Formula,
the magnitude of $z(x)$ is approximately $x^\gamma$, and
for the case of $\wt L=(B-1,\dots,B-1)$, the same magnitude is predicted in \cite{heuberger}.
As in \cite{heuberger}, we define $\delta(x):=z(x)/x^\gamma$, and transform
this ratio as a function on real numbers.

Recall from Definition \ref{def:periodic-real} and Theorem \ref{thm:real}  periodic \zec\ collections
for $\OI$, their maximal blocks, and 
their \funds s.
\begin{notation}
\rm \label{def:bbeta}
For $n\in\nat$,
let  $\bbeta^n$   and $\bar\theta^n$ be the maximal $L^*$-block  and $\wt L^*$-block  at index $n$,
respectively, and 
let $\omtilde = \phitilde\Inv$ and $\ome=\phi\Inv$.
Let $Q$ and $\wt Q$
denote  the \funds s of $\cE^*$ and $\cEtilde^*$.
\end{notation}

Recall Definition \ref{def:cE}.
For $\ep\in\cEtilde$ of order $n\ge 1$, we have $\bar\ep\in \cE$.
Suppose that $\bar\ep\ne \beta^{n+1}$, and recall the function $\rev_n$ from Definition \ref{def:restriction}. Then,
(\ref{eq:three-ratios}) implies
\begin{align}
\delta(x) &= \frac{\sum \bar\ep H }{(\sum \ep \wt H)^\gamma}
=\frac{\sum \bar\ep_k (\al \phi^{k-1} + O(\phi^{rk})) }
{(\sum \ep_k (\wt \al \phitilde^{k-1} + O(\phitilde^{rk})) )^\gamma}
\quad\text{where $r<1$ and $\textstyle\sum=\sum_{k=1}^n$}\notag\\
&=\frac{\left(\sum \bar\ep_k \al \phi^{k-1}\right) + O(\phi^{rn}) }
{\left(\left( \sum \ep_k \wt \al \phitilde^{k-1}\right) + O(\phitilde^{rn}) \right) ^\gamma}
	=\frac{\al \phi^{n }}{\wt \al ^\gamma \phitilde^{\gamma n}}
\cdot \frac{\sum \bar\ep_k \phi^{k-n-1} +   O(\phi^{rn-n })}
{\left(\sum \ep_k \phitilde^{k-n-1} + O(\phitilde^{rn-n})\right)^\gamma}\notag\\
&=\frac{\al }{\wt \al^\gamma }
\cdot \frac{\sum_{t=1}^n \bar\ep_{n+1-t} \ome^{t} + O(\ome^{(1-r)n }) }
{\left(\sum_{t=1}^n \ep_{n+1-t} \omtilde^{t} + O(\omtilde^{(1-r)n }) \right)^\gamma}
=\frac{\al }{\wt \al^\gamma }
\cdot \frac{\sum \rev_n(\bar\ep) Q + O(\ome^{n}) }
{\left(\sum \rev_n( \ep) \wt Q + O(\omtilde^n) \right)^\gamma}.
\label{eq:real-ratio}\\
\intertext{If $\bar\ep= \beta^{n+1}$, then}
\delta(x) &=\frac{\al }{\wt \al^\gamma }
\cdot \frac{1 + O(\ome^{(1-r)n }) }
{\left(\sum_{t=1}^n \ep_{n+1-t} \omtilde^{t} + O(\omtilde^{(1-r)n }) \right)^\gamma}
	=\frac{\al }{\wt \al^\gamma }
\cdot \frac{1 + O(\ome^n) }
{\left(\sum \rev_n(\ep) \wt Q + O(\omtilde^{n}) \right)^\gamma}.\label{eq:real-ratio-special}
\end{align}
Notice that $n\to\infty$ as $x\to\infty$, and hence, $O(\ome^n)=o(1)\to 0$.
Thus, as $x\to \infty$, 
the ratio in (\ref{eq:real-ratio}) approaches 
a value depending only on the real number $\rev_n(\ep) \wt Q\in\OI$.
This observation motivates us to define $\delta(x)$ as a function on $\OI$,
and it is introduced in Definition \ref{def:rev-bar} and Proposition \ref{prop:del-delstar} below. 

For that goal, we need to understand the relationship between
$\rev_n(\ep)$ and $\rev_n(\bar\ep)$ as \cf s in $\cEtilde^*$, which are 
apprearing in (\ref{eq:real-ratio}).
We begin with the transition of Lemma \ref{lem:ep-bar} for $\cEtilde^*$, which is Section \ref{sec:mixed},
and the relationship between the \cf s is explained in Definition \ref{def:rev-bar} in Section \ref{sec:transition}.

\subsubsection{Mixed decomposition}\label{sec:mixed}
 
Recall from Notation \ref{def:bbeta} that $\bbeta^n$ for $n\ge 1$ denote  the maximal $\Lstar$-blocks of $\cE^*$.

\begin{lemma}\label{lem:no-proper-blocks}
Let $\ep $ be a \cf\ not necessarily in $\cEtilde^*$ such that $\res_{b}(\ep)=0$ for an integer $b\ge 0$.
Then,
$\ep \dord \bbeta^{b+1} $
if and only if
$\res_n(\ep)$ is a proper $L^*$-block with support interval $[b+1,n]$ for some $n\ge b+1$.
\end{lemma}

\begin{proof}
Suppose that $\ep \dord \bbeta^{b+1} $ for an integer $b\ge 0$.
By the definition of the inequality there is an index $n\ge b+1$ such that $\res_{n-1}(\ep)=\res_{n-1}(\bbeta^{b+1})$ and
$\ep_n<\bbeta^{b+1}_n$, so $\res_n(\ep)$ is a proper $L^*$-block with support interval $[b+1,n]$.

Suppose that $\res_n(\ep)$ is a proper $L^*$-block with support interval $[b+1,n]$ for some $n\ge b+1$.
By the definition of a proper $L^*$-block with support interval $[b+1,n]$,
we have $\res_n(\ep) \dord \bbeta^{b+1}$,
$\res_n(\ep)_n <\bbeta^{b+1}_n$, and hence, $\ep \dord \bbeta^{b+1}$.
\end{proof}

In Corollary \ref{cor:LL-decomposition} below and on, we define
$\bbeta^\infty :=0$ for convenience.
\begin{cor}\label{cor:LL-decomposition}
Let $\ep\in\cEtilde^*$. Then, there is a (unique) largest integer $\ell \ge 0$ or $\ell=\infty$
such that $\ep=\sum_{m=1}^\ell \zeta^m + \mu$ and
$\sum_{m=1}^\ell \zeta^m$ is a proper $L^*$-block decomposition
where
the support interval of $\zeta^\ell$ is $[i,b]$,
$\res_b(\mu)=0$, and $\bbeta^{b+1}\dordeq \mu$;
if $\ell=0$, then $b=0$ and $\ep=\mu$, and if $\ell=\infty$, then $\mu=0$ and 
$\ep \in \cE^*$.
\end{cor}

\begin{proof}
Suppose that $\ep\in\cE^*$.
Then, we have the $\Lstar$-\bdecom\ $\ep= \sum_{m=1}^\infty \zeta^m$, and this corresponds
to the case $\ell=\infty$.

Suppose that $\ep\not\in\cE^*$.
Then, there must be a unique largest integer $\ell\ge 0$ such that
$\sum_{m=1}^\ell \zeta^m $ is a proper $L^*$-block decomposition where $[i,b]$ is the support interval of $\zeta^\ell$,
$\ep=\sum_{m=1}^\ell \zeta^m + \mu$, and $\res_b(\mu)=0$.
As in Lemma \ref{lem:ep-bar}, the maximality of $\ell$ implies that $\res_n(\mu)$ is not a proper $\Lstar$-block for any $n\ge b+1$,
and hence, by Lemma \ref{lem:no-proper-blocks},
we find $\bbeta^{b+1}\dordeq \mu$.

\end{proof}
\begin{deF}\label{def:LL-decomposition}\rm
Given $\ep\in\cEtilde^*$, an expansion $\ep=\sum_{m=1}^\ell \zeta^m + \mu$ where
$\zeta^m$, $\ell$, and $\mu$ are as described in Corollary \ref{cor:LL-decomposition}
is called {\it an $\Lsbar$-block decomposition of a member of $\cEtilde^*$} where
$\zeta^\ell$ is uniquely determined, and
if $[i,b]$ is the \supint\ of $\zeta^\ell$, the index $b$ is called
{\it the largest $\Lsbar$-support index of $\ep$};  $b=\infty$ if $\ep\in\cE^*$, 
and $b=0$ if $\ep=\mu$.
Let $\overline{\cE^*}$ denote the union of $\cE^*$ and
the set of \cf s $\sum_{m=1}^\ell \zeta^m + \bbeta^{b+1}$ where
$\zeta^m$ and $\mu=\bbeta^{b+1}$ are as described in Corollary \ref{cor:LL-decomposition}.

\end{deF}

The following lemma and corollary show that an $\wt L^*$-\bdecom\ of an
$\Lsbar$-\bdecom\ maintains the boundaries of the \supint s of the proper $\Lstar$-blocks.

\begin{lemma}\label{lem:Ltilde-decom}
Let $\zeta$ be a non-zero proper $\Lstar$-block with \supint\ $[i,b]$.
Then, there is an $\wt L^*$-\bdecom\ $\zeta=\sum_{m=1}^t \eta^m$
such that the \supint\ of $\eta^t$ is $[a,b]$.
\end{lemma}

\begin{proof}
Let $\zeta$ be a non-zero proper $\Lstar$-block with \supint\ $[i,b]$.
Let $\bbeta^i=\sum_{m=1}^\infty \eta^m$ be the full $\wt L^*$-\bdecom\ where
$[i_m,c_m]$ is the \supint\ of $\eta^m$, so $b$ is contained in a
unique interval
$[i_t,c_t]$.
Notice that
$\res_{b-1}(\zeta) = \res_{b-1}(\bbeta^i)$ and
$\zeta_b<\bbeta^i_b=\eta^{i_t}_b\le\ddelta^{i_t}_b$.
It follows that $\zeta_b < \ddelta^{i_t}_b$, and hence, $\eta^0:=\res_b^{i_t}(\zeta)$
is a proper $\wt L^*$-block with \supint\ $[i_t,b]$.
Thus, $\zeta = \sum_{m=1}^{t-1} \eta^m + \eta^0$ is an $\wt L^*$-\bdecom.
\end{proof}

\begin{cor}\label{cor:mu}
Let $\ep=\sum_{m=1}^\ell \zeta^m + \mu$ be an $\Lsbar$-\bdecom\ of $\ep\in \cEtilde^*$.
Then, $\mu\in\cEtilde^*$.
\end{cor}
\begin{proof}
If $[i,b]$ is the \supint\ of $\zeta^\ell$, by Lemma \ref{lem:Ltilde-decom},
the (disjoint) union of the \supint s of an $\wt L^*$-\bdecom\ of $\sum_{m=1}^\ell \zeta^m$ is contained
in $[1,b]$. Thus, $\mu\in\cEtilde$ has an $\wt L^*$-\bdecom, and hence, it is a member of
$\cEtilde^*$.

\end{proof}

\subsubsection{Transition to a function on $\OI$} \label{sec:transition}

Given $\ep\in \cEtilde^*$ and an integer $n\ge 1$,
notice that
$\rev_n(\ep)\in \cEtilde$ and $\overline{ \rev_n(\ep)}\in \cE$, which is defined in
Definition \ref{def:cE}.
\begin{lemma}\label{lem:ep-bar-real}
Suppose that $\ep\in\cEtilde^* $ and $\ep\dord \bbeta^1$.
If $\ep\in \cE^*$, then $\rev_n\left(\,\overline{ \rev_n(\ep)}\, \right) = \res_n(\ep)$.
If $\ep\in \cEtilde^* -\cE^*$, i.e.,
$\ep=\sum_{m=1}^\ell \zeta^m + \mu$ is an $\Lsbar$-block decomposition
with the largest $\Lsbar$-support index $b<\infty$
where the integer $\ell\ge 1$ and
$b\ge 1$,
then for sufficiently large $n$,
$$\rev_n\left(\,\overline{ \rev_n(\ep)}\, \right)=
\begin{cases}
\sum_{m=1}^\ell \zeta^m + \beta^b &\text{if }\bbeta^{b+1}\dord \mu,\\
\res_n(\ep) &\text{if }\bbeta^{b+1}= \mu.
\end{cases}
$$
\end{lemma}

\begin{proof}
If $\ep\in \cE^*$, then
$\rev_n(\ep)\in \cE$, so $\rev_n\left(\,\overline{ \rev_n(\ep)}\, \right) =\rev_n\left(\, \rev_n(\ep) \, \right) = \res_n(\ep)$.
Suppose that $\ep\in \cEtilde^* -\cE^*$, so we have the $\Lsbar$-block decomposition as described in
the lemma, and notice that
by Lemma \ref{lem:no-proper-blocks}, $\ep\dord \bbeta^1$ implies $b\ge 1$.
Let
$n\ge b+1$.
%
If $\bbeta^{b+1}\dord \mu$,
then there is a smallest index $p\ge b+1$ such that $\bbeta^{b+1}_p< \mu_p$.
So,
if $n\ge p$, then
$\rev_n(\sum_{m=1}^\ell \zeta^m + \bbeta^{b+1}) \aord\rev_n(\sum_{m=1}^\ell \zeta^m + \mu) $, and
by applying Lemma \ref{lem:ep-bar} to $\rev_n(\ep)=\rev_n(\sum_{m=1}^\ell \zeta^m + \mu)$, we have
$ \overline{\rev_n(\ep)}
= \rev_n\left(\sum_{m=1}^\ell \zeta^m + \beta^b \right)$, i.e.,
$\rev_n\left(\overline{\rev_n(\ep)}\right)
= \res_n(\sum_{m=1}^\ell \zeta^m + \beta^b)=\sum_{m=1}^\ell \zeta^m + \beta^b$.
If $\bbeta^{b+1}=\mu$, then
$\rev_n(\ep)= \rev_n(\sum_{m=1}^\ell \zeta^m + \bbeta^{b+1})\in \cE$, and
hence, $\overline{\rev_n(\ep)}=\rev_n(\ep)$.
Thus, $\rev_n\left(\,\overline{ \rev_n(\ep)}\, \right)=\rev_n\left( \rev_n(\ep) \, \right)
=\res_n(\ep)$.
\end{proof}

Let us use Lemma \ref{lem:ep-bar-real} to define $\revbar(\ep)$ for $ \ep \in\cEtilde^*$,
which is a version of Lemma \ref{lem:ep-bar} for $\OI$.
\begin{deF}\label{def:rev-bar}\rm
Given $\ep\in \cEtilde^*$,
if $\bbeta^1\dordeq \ep$, then define
$\revbar(\ep): =\bbeta^1$, and if $\ep\dord \bbeta^1$,
define $\revbar(\ep): = \lim_{n\to\infty}\rev_n\left(\,\overline{ \rev_n(\ep)}\, \right)$.
Let $\delta_1^* : [\wt \ome,1)\to \real$ be the function given by
$$ \delta_1^*(\tsum \ep \wt Q):=\dis
\frac{\sum \revbar(\ep) Q }
{\left(\sum \ep \wt Q \right)^\gamma}
\quad\text{if $\omtilde\le \sum\ep \wt Q<1$ for $\ep\in\cEtilde^*$.}
$$
Given a real number $y\in \OI$, there is a unique positive integer $n$ such that $y\in [\ome^n, \ome^{n-1})$.
Let $\delta^*:  \OI \to \real$ be the function given by $\delta^*(y) = \delta_1^*(y/\ome^{n-1})$ where $n$ is defined as above.
We also denote $\delta^*(\sum\ep \wt Q)$ simply by $\delta^*[\ep]$ for all $\ep\in\cEtilde^*$.

\end{deF}
If $\ep \dord \bbeta^1$ and $\ep\in \cEtilde^*$,
then by Lemma \ref{lem:ep-bar-real}, the limit in Definition \ref{def:rev-bar} is well-defined.
By Theorem \ref{thm:real}, given a real number $y\in [\ome,1)$, there is a unique $\ep\in \cEtilde^*$
such that $y=\sum\ep \wt Q$, so $\delta_1^*$ is well-defined.

In Proposition \ref{prop:del-delstar}, we compare the values of $\delta$ with those of $\delta^*$, and 
 we begin with a few preliminaries:
\begin{lemma}\label{lem:ep-bar-ep}
Let $\ep\in\cEtilde^*$.
Then, $\ep\in\cEsbar$
if and only if
$\revbar(\ep)=\ep$.
\end{lemma}

\begin{proof}
Suppose that $\bbeta^1\dordeq \ep$.
Then, by Lemma \ref{lem:no-proper-blocks} and Corollary \ref{cor:LL-decomposition}, $\ep\in\cEsbar$ if and only if $\ep=\bbeta^1$.
If $\ep\in\cEsbar$, then $\ep=\bbeta^1$, and by Definition \ref{def:rev-bar}, $\revbar(\ep)=\ep$.
If $\ep\not\in\cEsbar$, then $ \bbeta^1\dord \ep$ by Corollary \ref{cor:LL-decomposition}, and
by Definition \ref{def:rev-bar}, $\revbar(\ep)=\bbeta^1 \ne \ep$.

Suppose that $\ep \dord \bbeta^1$.
If $\ep\in \cEsbar - \cE^*$, then, by Corollary \ref{cor:LL-decomposition},
we have an $\Lsbar$-\bdecom\ $\ep = \sum_{m=1}^\ell \zeta^m + \bbeta^{b+1}$
where $b$ is the largest $\Lsbar$-support index of $\ep$ (see Definition \ref{def:LL-decomposition}),
and by Lemma \ref{lem:no-proper-blocks}, $\ep \dord \bbeta^1$ implies that $b\ge 1$.
If $\ep\in\cEsbar$, i.e., $\ep\in\cE^*$ or $\ep = \sum_{m=1}^\ell \zeta^m + \bbeta^{b+1}$,
then by Corollary \ref{cor:LL-decomposition}, $1\le \ell\le \infty$, and by Lemma \ref{lem:ep-bar-real},
$\rev_n\left(\,\overline{ \rev_n(\ep)}\, \right)
=\res_n(\ep)$ for sufficiently large $n$, which implies that $\revbar(\ep)= \ep$.
If $\ep\not\in\cEsbar$, i.e.,
$\ep = \sum_{m=1}^\ell \zeta^m + \mu$ where $\bbeta^{b+1}\dord \mu$,
then by Lemma \ref{lem:ep-bar-real},
$\rev_n\left(\,\overline{ \rev_n(\ep)}\, \right)
= \sum_{m=1}^\ell \zeta^m + \beta^b$ for sufficiently large $n$.
Thus,
$\revbar(\ep)=\sum_{m=1}^\ell \zeta^m + \beta^b >_\textrm{d} \ep$;
in particular, $\revbar(\ep)\ne \ep$.
\end{proof}

In Lemma \ref{lem:rev-rev}, we show that $\revbar$ extends the map
$\rev_n( \ep)\mapsto \rev_n(\bar \ep)$ appearing in 
(\ref{eq:real-ratio}).
Recall Definition \ref{def:cE}.
\begin{lemma}\label{lem:rev-rev}
Suppose that $\mu\in \cEtilde$, and $t:=\ord(\mu)>0$.
If $\mu \aordeq \hbeta^{t+1}$, then
$ \revbar(\rev_t(\mu)) =\rev_t(\bar\mu)$.
If $ \hbeta^{t+1}\aord \mu$, then $\revbar(\rev_t(\mu))=\bbeta^1$.
\end{lemma}

\begin{proof}
Suppose that $\mu\aordeq \hbeta^{t+1}$.
If $\mu\in\cE$, then $\bar\mu=\mu$, and by Lemma \ref{lem:rev-ep}, $\rev_t(\mu)\in\cE^*$.
Thus, by Lemma \ref{lem:ep-bar-ep}, $\revbar(\rev_t(\mu))=\rev_t(\mu)=\rev_t(\bar\mu)$.

Suppose that $\mu\in\cEtilde-\cE$ and $\mu \aord \hbeta^{t+1}$. 
By Lemma \ref{lem:ep-bar}, $ \mu =\zeta^0+ \sum_{m=1}^n \zeta^m$ 
where $\zeta^0$ and $\zeta^m$ for $1\le m\le n$ 
are as defined in the lemma,  and $[a,i]$ is the support interval of $\zeta^1$,
 so that $\hbeta^{a-1}\aord\zeta^0$,
and $n\ge 1$.
Then, 
$\bar\mu = \beta^a+ \sum_{m=1}^n \zeta^m$.
Let $\ep:=\rev_t(\mu)\in\cEtilde^*$.
Then, $\ep \in \cEtilde^* - \cE^*$, and $\ep\dord \bbeta^t \dordeq \bbeta^1$.
By Lemma \ref{lem:ep-bar-real},
for sufficiently large $n$, $\rev_n\left(\,\overline{ \rev_n(\ep)}\, \right) =
\rev_t\left( \beta^a+ \sum_{m=1}^n \zeta^m \right)
=\rev_t(\bar\mu)$.
Thus, $\revbar(\ep)=\rev_t(\bar\mu)$, i.e., $\revbar(\rev_t(\mu))=\rev_t(\bar\mu)$.

Suppose that $ \hbeta^{t+1}\aord \mu$.
Then, there is an index $1\le p \le t$ such that $\hbeta^{t+1}_p<\mu_p$
and $\res^{p+1}(\hbeta^{t+1})=\res^{p+1}(\mu)$.
Thus, $\res_t(\bbeta^{1})\dord \rev_t(\mu)$, and
hence, $ \bbeta^{1} \dord \rev_t(\mu)$.
By definition, $\revbar(\rev_t(\mu))=\bbeta^1$.

\end{proof}

\begin{prop}\label{prop:del-delstar}
Given a positive integer $x$,
let $x=\sum\mu\wt H$ where $\mu\in\cEtilde$ and $\ord(\mu)=n$.
Then,
$ \delta(x) \sim \delta^*(\sum\rev_n(\mu)\wt Q)\al/\wt\al^\gamma$ as $x\to\infty$.
\end{prop}
\begin{proof}
If $\mu \aordeq \hbeta^{n+1}$, then by (\ref{eq:real-ratio}) and Lemma \ref{lem:rev-rev},
\begin{equation*}
\delta(x)=\frac{\al }{\wt \al^\gamma }
\cdot \frac{\sum \rev_n(\bar\mu) Q + O(\ome^{n}) }
{\left(\sum \rev_n( \mu) \wt Q + O(\omtilde^n) \right)^\gamma}
\sim \frac{\al }{\wt \al^\gamma }
\cdot \frac{\sum \revbar(\rev_n( \mu))Q  }
{\left(\sum \rev_n( \mu) \wt Q \right)^\gamma}
=\delta^*(\sum\rev_n(\mu)\wt Q)\,\frac{\al }{\wt \al^\gamma }.
\end{equation*} 
If $\hbeta^{n+1} \aord \mu$, then
$\bar\mu = \beta^{n+1}$, and by Lemma \ref{lem:rev-rev} and (\ref{eq:full-recursion-real}),
$$\delta^*(\tsum\rev_n(\mu)\wt Q)\dis
= \frac{\sum\revbar(\rev_n(\mu) ) Q }{\left(\sum\rev_n(\mu) \wt Q\right)^\gamma }
= \frac{\sum\bbeta^1 Q }{\left(\sum\rev_n(\mu) \wt Q\right)^\gamma }
= \frac{1 }{\left(\sum\rev_n(\mu) \wt Q\right)^\gamma }.
$$
By (\ref{eq:real-ratio-special}), the quotient of
$ \delta(x)$ and $ \delta^*(\sum\rev_n(\mu) \wt Q)\al/\wt\al^\gamma$ approaches $1$.
\end{proof}

Notice that for large values of positive integers $x=\sum\mu \wt H$ and $\ord(\mu)=n$,
i.e., $x=\wt \al \sum_{k=1}^n \mu_k \wt\phi^{k-1} + O(\wt\phi^{nr})$ for $0<r<1$,
we have
$\sum\rev_n(\mu) \wt Q\sim \wt \phi^{\set{\log_{\wt \phi}(x/\wt \al)}-1}$.
So, we may state the asymptotic relation in Proposition \ref{prop:del-delstar} as follows:
\begin{equation}\label{eq:delta*-x}
\delta(x) \sim \frac\al{\wt \al^\gamma}\, \delta^*\left( \wt \phi^{\set{\log_{\wt \phi}(x/\wt \al)}-1} \right).
\end{equation}

\section{Proof of the main result}\label{sec:proofs}

We shall prove the main results in this section.
Assume the notation and context of Section \ref{sec:zec-exressions}.
Recall from Theorem \ref{thm:ratios} that $\wt\ome < \ome$ and $\gamma=\log_{\phitilde}\phi<1$,
and from Theorem \ref{thm:real}, the evaluation map $\eval_{\wt Q} : \cEtilde^* \to \OI$, which
is increasing and bijective.
Also recall the subset $\cEsbar$ of $\cEtilde^*$  from Definition \ref{def:LL-decomposition}.
For the remainder of the paper, we use the following notation:
\begin{notation}
\rm \label{notation:U}
Let $\cE^*\fin$ denote the subset of $\cE^*$ consisting of \cf s that have finite support,
and let $U$ denote the image $\eval_{\wt Q}(\cEtilde^*-\cEsbar)$ in $\OI$.
\end{notation}

\subsection{Continuity and differentiability}\label{sec:cont-diff}

For the case of $\wt L=(B-1,\dots,B-1)$, the asymptotic formula introduced in
\cite{heuberger} implies that $\delta^*$ is continuous.
For the general cases considered in this paper, $\delta^*$ is continuous as well,
and differentiable precisely on $U$ which
has Lesbegue measure $1$.
Let us prove these facts.

\subsubsection{Continuity}
Let us first understand the behavior of $\revbar$ in the \lq\lq neighborhood\rq\rq\ of $\ep\in\cEtilde^*$.
Recall $\inF_{\cEtilde^*}$ from Definition \ref{def:inf-version},
and let $\inF:=\inF_{\cEtilde^*}$.  For example, $\inf(\beta^1) = \ddelta^2$, and 
by (\ref{eq:full-recursion-real}), we have
$\sum \beta^1 \wt Q=\sum \inf(\beta^1) \wt Q =\sum \ddelta^2 \wt Q$.
\begin{lemma}\label{lem:ep+-}
Let $x=\sum\ep\wt Q>0$ for $\ep\in\cEtilde^*$, and
let $\Delta x$ be a sufficiently small positive real number.
Let $x+\Delta x = \sum\ep^+\wt Q$ and
$x-\Delta x = \sum\ep^-\wt Q$ where
$\set{\ep^+,\ep^-}\subset \cEtilde^*$.

For the following statements, $n$ represents an index determined by
$\Delta x$ such that $n\to\infty$ as $\Delta x\to 0$.
If $\ep\in \cEtilde^*-\cEsbar$, then $ \revbar(\ep) = \revbar(\ep^+) =\revbar(\ep^- )$.
If $\ep\in \cEsbar - \cE^* $, then
$ \revbar(\ep)
= \inF_{\cE^*}(\revbar(\ep^+))$, and
$\res_n(\revbar(\ep)) =\res_n(\revbar(\ep^-))$.
If $\ep\in \cE^* $, then $\res_n(\revbar(\ep))=\res_n(\revbar(\ep^+))=\res_n(\revbar(\ep^-)) $.
\end{lemma}

\begin{proof}
Let $\ep=\sum_{m=1}^\ell \zeta^m + \mu$ be an
$\Lsbar$-\bdecom\ with the largest $\Lstar$-block support index $b$.
Recall that, by Lemma \ref{lem:ep-converge},
there is $n\in \nat$ determined by the sufficiently small $\Delta x$ such that 
$\res_n(\ep)=\res_n(\ep^+)$ and $\res_n(\inF(\ep))=\res_n( \ep^- )$, and that $n\to\infty$ as $\Delta x\to 0$.
Also recall the meaning of $\Lstar$- or $\wt L^*$-\decom s from Definition \ref{def:periodic-real}.
For the remainder of the proof, all the expressions of $\ep$, $\ep^+$, $\ep^-$, $\inF(\ep)$, and
$\inF(\ep^+)$ are $\tLstar$-\decom s %
unless specified differently.

Suppose that $\ep\in \cEtilde^*-\cEsbar $, i.e., $\bbeta^{b+1}\dord \mu$.
By Theorem \ref{thm:real},
$\sum_{m=1}^\ell \zeta^m +\bbeta^{b+1} \dord \ep^- \dord \ep \dord \ep^+$
for sufficiently small $\Delta x$.
By Lemma \ref{lem:ep-converge}, we have
$\ep^+=\sum_{m=1}^\ell \zeta^m + \tau^+$.
Then, 
the earlier inequalities and the \lex\ order imply $\bbeta^{b+1}\dord \tau^+$, and 
also $\ep^-=\sum_{m=1}^\ell \zeta^m + \tau^-$ where $\bbeta^{b+1}\dord \tau^-$.
Thus, both expressions are $\Lsbar$-block decompositions, 
and by Lemma \ref{lem:ep-bar-real}, $\revbar(\ep^+)=\revbar(\ep^-)=\revbar(\ep)=\sum_{m=1}^\ell \zeta^m +\beta^b$.

Suppose that $\ep\in \cEsbar -\cE^*$, i.e., $\mu=\bbeta^{b+1} $.
So, by Lemma \ref{lem:ep-converge}, 
$\ep^+=\sum_{m=1}^\ell \zeta^m + \tau^+$ where $\bbeta^{b+1}\dord \tau^+$.
Since $\ep$ does not have finite support, we have $\inF(\ep )=\ep $,
and, by Lemma \ref{lem:ep-converge}, $\ep^-=\sum_{m=1}^\ell \zeta^m + \tau^-$ where $ \tau^-\dord\bbeta^{b+1}$.
Since $\bbeta^{b+1}\dord \tau^+$, as in the earlier case,
$\revbar(\ep^+)=\sum_{m=1}^\ell \zeta^m +\beta^b$,
and $\inF_{\cE^*}(\revbar(\ep^+))=\ep=\revbar(\ep)$.
Notice that, by Lemma \ref{lem:no-proper-blocks},  $ \tau^-\dord\bbeta^{b+1}$ implies that
$\tau^-$
has an $\Lsbar$-\bdecom\ $\sum_{m=\ell+1}^s \zeta^m + \mu^-$ where $ \ell+1\le s \le \infty$, and
let $[b+1,c]$ be the \supint\ of the proper $\Lstar$-block $\zeta^{\ell+1}$.
From Lemma \ref{lem:ep-converge}, we have
$\res_n(\inF(\ep))=\res_n( \ep )=\res_n( \ep^- )$ for $n\to\infty$,
which implies that $ \res_n(\bbeta^{b+1} )=\res_n( \tau^- )$.
Since we have $\zeta^{\ell+1}_c<\bbeta^{b+1}_c$ from $\zeta^{\ell+1}$ being a proper $\Lstar$-block
at index $b+1$,
we have $c>n$, and hence, $ c\to\infty$ as $n\to\infty$.
By considering all cases of $\mu^-$ for Lemma \ref{lem:ep-bar-real},
$\res_{c-1 }(\revbar(\ep^-))=\res_{c-1}(\ep)=\res_{c-1}(\revbar(\ep))$.

Suppose that $\ep\in \cE^*-\cE^*\fin$, i.e., $\mu=0$ and $\ell=\infty$.
So, we have $\inF(\ep )=\ep $, and by Lemma \ref{lem:ep-converge}, 
$\ep^+=\sum_{m=1}^T \zeta^m + \tau^+$
and $\ep^-=\sum_{m=1}^T \zeta^m + \tau^-$ where $T\to\infty$ as $\Delta x\to 0$.
For all cases of $\tau^+$ and $\tau^-$ for Lemma \ref{lem:ep-bar-real},
Corollary \ref{cor:mu} implies
$\revbar(\ep^+)=\sum_{m=1}^{T-1} \zeta^m + \eta^+$
and
$\revbar(\ep^-)=\sum_{m=1}^{T-1} \zeta^m + \eta^-$
are $\tLstar$-decompositions.
Since $T\to\infty$,
we have $\res_n(\revbar(\ep))=\res_n(\revbar(\ep^+))=\res_n(\revbar(\ep^-)) $.

Suppose that $\ep\in \cE^*\fin$, i.e., $\mu=0$, $\ell=\infty$, and $\zeta^m=0$ for all sufficiently large $m$.
Let $T$ be the largest index such that $\zeta^T\ne 0$, and 
let $c$ be the largest index such that $\zeta^T_c\ge 1$.
By Lemma \ref{lem:ep-converge}, we have
$\ep^+=\sum_{m=1}^T \zeta^m + \tau^+$
and $\ep^-=\sum_{m=1}^T \zeta^m -\beta^c + \tau^-$
where there is sufficiently large $n\in\nat$ such that $\res_{n}(\tau^+)=0$,
and $\res_n(\tau^-)=\res_n(\ddelta^{c+1})$, which implies 
$\bbeta^{c+1}\dord \tau^-\dord \beta^{c}$.
For all cases of $\tau^+$ and $\tau^-$, by Corollary \ref{cor:mu} and Lemma \ref{lem:ep-bar-real},
$\revbar(\ep^+)=\sum_{m=1}^T \zeta^m + \revbar(\tau^+)$
and
$\revbar(\ep^-)=\sum_{m=1}^T \zeta^m-\beta^c +\beta^c=\ep$.
By Lemma \ref{lem:ep-bar-real} applied to all cases of $\tau^+$,
there is sufficiently large $n\in\nat$ such that
 $\res_n(\revbar(\ep^+))=\sum_{m=1}^T \zeta^m=\ep=\revbar(\ep)=\res_n(\revbar(\ep))$
 and 
 $\revbar(\ep^-)= \ep=\res_n(\revbar(\ep))$.
This proves the assertion.

\end{proof}

Recall the subset $U$ from Notation \ref{notation:U}.
\begin{theorem}\label{thm:continuous}
The function $\delta^*$ is continuous, and
locally decreasing and differentiable on $U$.
\end{theorem}

\begin{proof}
Let us prove that $\delta_1^*$, defined in Definition \ref{def:rev-bar}, is continuous.
Let $\ep\in\cEtilde^*$ such that $ \beta^1\dord \ep$, and let $x=\sum \ep \wt Q$ and
$y= \sum \revbar(\ep) Q$.
Let $\Delta x$, $\ep^+$, and $\ep^-$ be as defined in Lemma \ref{lem:ep+-}.

Suppose that $\ep\in \cEtilde^*- \cEsbar$.
Then, by Lemma \ref{lem:ep+-},
$\delta^*(x+\Delta x) = \sum\revbar(\ep )\wt Q/(x+\Delta x)^\gamma$
and $\delta^*(x-\Delta x) = \sum\revbar(\ep )\wt Q/(x-\Delta x)^\gamma$.
Since the numerator remains constant, $\delta^*$ is decreasing and differentiable on an open neighborhood of $x$; in particular, it is continuous at $x$.
This proves the assertion about $U$.

Suppose that $\ep\in \cEsbar - \cE^*$.  Notice that
given $\tau\in\cEsbar$, by (\ref{eq:full-recursion-real}),
$\sum \tau Q = \sum \inf_{\cE^*}(\tau) Q$.
Since $\revbar (\ep^+ )\in \cEsbar$,
by Lemma \ref{lem:ep+-}, $\sum\revbar (\ep^+ )Q =\sum\inf_{\cE^*}(\revbar (\ep^+ ))Q
= \sum\revbar(\ep)Q$.
Thus, $\sum\revbar(\ep^+)Q /(x+\Delta x)^\gamma=\sum\revbar(\ep)Q /(x+\Delta x)^\gamma
\to \sum\revbar(\ep)Q /x^\gamma $ as $\Delta x\to 0$.
By Lemma \ref{lem:ep+-},
 there is sufficiently large $n$ such that 
$\res_n(\ep)=\res_n(\revbar(\ep))=\res_n(\revbar(\ep^-))$, and 
let $\tau: =\res^{n+1}(\revbar (\ep^- ))$, so that
$ \revbar (\ep^- )=\res_n(\revbar (\ep^- )) + \tau$.
By Lemma \ref{lem:ep+-}, $\sum\revbar (\ep^- )Q
= \sum(\res_n (\ep^- )+ \tau)Q
= \sum \res_n(\revbar (\ep) )Q + \sum\tau Q$.
Notice that $\revbar(\ep^-)\in \cEsbar$ implies $\revbar(\ep^-)_k\le M$ for all $k\in\nat$
where $M$ is the largest entry in $L$, and hence,
$\sum\tau Q =\sum_{k=n+1}^\infty \tau_k \ome^k\le\sum_{k=n+1}^\infty M \ome^k
\to 0 $ as $n\to\infty$. 
By Lemma \ref{lem:ep+-}, $\sum\revbar (\ep^- )Q
= \sum(\res_n (\ep^- )+ \tau)Q
= \sum \res_n(\revbar (\ep) )Q + \sum\tau Q$. 
Thus, $\sum\revbar (\ep^- )Q \to \sum \revbar (\ep) Q $ as $\Delta x\to 0$, which proves
that $\sum\revbar(\ep^-)Q /(x-\Delta x)^\gamma\to \sum\revbar(\ep)Q /x^\gamma $ as $\Delta x\to 0$.

It remains to show that $\delta_1^*[\ep^{\pm}]\to \delta_1^*[\ep]$ where $\ep\in\cE^*$ with $\beta^1\dord \ep$
and that $\delta_1^*[\ep^+]\to \delta_1^*[\ep]$ where $\ep=\beta^1$.
The proofs of these cases are similar to the earlier cases, and
we leave it to the reader.
This
concludes the proof of $\delta_1$ being continuous on $[\wt\ome,1)$.

Let us show $\delta_1^*(\wt \ome) = \lim_{x\to 1^-}\delta_1^*(x)=1$, so
it demonstrates that
when we define $\delta^*$ via $\delta^*_1$, the function values of those end points coincide.
Notice that $\delta^*(\wt \ome) = \ome/\wt \ome^\gamma=1$.
Recall that $\bar\theta^1$ is the maximal $\wt L^*$-block at index $1$, and $\sum\bar\theta^1\wt Q=1$.
If $\sum\res_n(\ddelta^1) \wt Q<x:=\sum\ep \wt Q<1$ where $n$ is sufficiently large and $\ep\in\cEtilde^*$, then
by Theorem \ref{thm:real},
	$\res_{n}(\bar\theta^1) = \res_{n}(\ep)$.
	By (\ref{eq:full-recursion-real}),
$$
\abs{1 - \sum \ep \wt Q }=\abs{\sum\bar\theta^1\wt Q - \sum \ep \wt Q }
=\abs{ \sum_{k=n+1}^\infty (\bar\theta^1_k-\ep_k) \wt\ome^k}
=O(\wt\ome^{ n})\to 0\ \text{ as }n \to\infty.$$
Thus, $\delta$ is continuous on $(0,1)$.
\end{proof}

\subsubsection{Differentiability}
Let us prove about the differentiability of $\delta^*$ on $\eval_{\wt Q}(\cEsbar)$.
\begin{lemma}
\label{lem:r}
Let $\ep\in \cEsbar - \cE^*\fin$.
Then, there are infinitely many indices $r$ such that 
$\ep - \beta^r \in \cEsbar$.
\end{lemma}

\begin{proof}
Let $\ep=\sum_{m=1}^\infty \zeta^m\in \cE^*-\cE^*\fin$ be
a non-zero
$L^*$-\bdecom.
Then, for each
$m$,
there is a largest index $r$ such that $[i,b]$ is the support interval of $\zeta^m$
and $\zeta^m_r\ne 0$.
Thus, $\zeta^m - \beta^r$ is a proper $L^*$-block with support interval $[i,r]$, and
$\sum_{m=1}^\infty \zeta^m - \beta^r\in \cE^*$.

Let
$\ep=\sum_{m=1}^\ell \zeta^m+\bbeta^{b+1}\in \cEsbar-\cE^*$
be an $\Lsbar$-\bdecom. 
Recall that $ L=(  e_1,\dots,   e_N)$ is the list,
by which the period \zec\ collection $\cE^*$ is determined, and 
 there is the largest index $r_0\le N$ such that $  e_{r_0}>0$.
 Then,  for all large $r \equiv b+r_0 \moD N$,
 $$\ep - \beta^r=\sum_{m=1}^\ell \zeta^m
 +\left(\res_r(\bbeta^{b+1}) - \beta^r\right) + \res^{r+1}(\bbeta^{b+1}) 
 = \sum_{m=1}^\ell \zeta^m
 +\left(\res_r(\bbeta^{b+1}) - \beta^r\right) + \bbeta^{r+N-r_0+1}.  $$
Since $\res_r(\bbeta^{b+1}) - \beta^r$ is a proper $\Lstar$-block
with support interval  $[b+1,r]$, 
the above decomposition implies that 
$\ep-\beta^r \in \cEsbar$. 
\end{proof}

\begin{theorem}\label{thm:not-diff}
The function $\delta^*$ is not differentiable on $\eval_{\wt Q}(\cEsbar)$.
\end{theorem}

\begin{proof}

Let $\ep\in \cEsbar - \cE^*\fin$, and let $x:=\sum\ep \wt Q$ and $y: = \sum \ep Q$.
By Lemma \ref{lem:r}, there is  sufficiently large $r$ such that $\ep -\beta^r\in \cEsbar$, and  let $\Delta x:=\wt Q_r=\wt\ome^r$, and let $\Delta y: = Q_r=\ome^r$.
Then, by Lemma \ref{lem:ep-bar-real} and the linear approximation of $1/x^\gamma$,
\begin{align*}
\frac1{-\Delta x}&\left( \delta^*(x-\Delta x) - \delta^*(x) \right)
= \frac1{-\wt\ome^r}\left( \frac{y -\ome^r}{(x-\wt\ome^r)^\gamma}
- \frac y{x^\gamma}\right) \\
&
=\frac1{-\wt\ome^r}\left( (y -\ome^r)\left( \frac1{x^\gamma}
	+ \frac{\gamma}{x^{\gamma+1}} \wt\ome^r
	 + O(\wt\ome^{2r})\right)
	- \frac y{x^\gamma}\right)
=-\frac{y\gamma}{x^{\gamma+1}} + O(\wt \ome^r)+
	\frac{\ome^r}{\wt\ome^r}\,\left( \frac1{x^\gamma}
	 + O(\wt\ome^{r})\right).
\end{align*}
As $r\to\infty$, the ratio $(\ome/\wt \ome)^r \to\infty$, and hence,
$\delta^*$ is not differentiable at $x\in \eval_{\wt Q}( \cEsbar-\cE^*\fin)$.
If $ \ep=\sum_{m=1}^\ell \zeta^m\in \cE^*\fin$ where $\ell$ is a positive integer,
a similar calculation for $\ep + \beta^r\in \cE^*\fin$ for sufficiently large $r$ shows
that the difference quotient approaches $+\infty$ as well, and hence,
$\delta^*$ is not differentiable at $x\in\eval_{\wt Q}( \cE^*\fin)$.

\end{proof}

Thus, $\delta^*$ is differentiable precisely on the subset $U $ defined in Notation \ref{notation:U}.

\subsubsection{Lesbegue meausre}
Let us prove that $U$ is a \lesb\ measurable subset of measure $1$.

\begin{prop}\label{prop:measurable}
Given an $\Lsbar$-\bdecom\ $\ep =\sum_{m=1}^\ell \zeta^m + \bbeta^{b+1}\in \cEsbar - \cE^*$,
let $x_\ep:=\sum\ep \wt Q$ and $x_\ep':=\sum(\sum_{m=1}^\ell \zeta^m+\beta^b)\wt Q
\in \cE^*\fin$. Then,
the subset $U$ is the disjoint union of the open intervals $(x_\ep,x_\ep')$ where $ \ep\in \cEsbar - \cE^* $;
in particular, $U$ is open with respect to the usual topology on $\OI$, and it is
Lesbegue measurable.
\end{prop}

\begin{proof}

Let us show that the open interval $(x_\ep,x_\ep')$ is contained in $U$.
Let $x=\sum\tau \wt Q$ where $x\in (x_\ep,x_\ep')$.
By Theorem \ref{thm:real}, $\sum_{m=1}^\ell \zeta^m +\bbeta^{b+1} \dord \tau \dord \sum_{m=1}^\ell \zeta^m + \beta^b$.
This implies that  $\tau=\sum_{m=1}^\ell \zeta^m + \mu $ where
$\bbeta^{b+1}\dord \mu $. Thus, the sum is the $\Lsbar$-\bdecom\ of $\tau\in \cEtilde^*-\cEsbar$, i.e., $x\in U$.

By Corollary \ref{cor:LL-decomposition},
given $x=\sum\tau\wt Q\in U$,
there is a unique full $\Lsbar$-\bdecom\ of $\tau=\sum_{m=1}^\ell \zeta^m + \mu $ where
$b$ is the largest $\Lsbar$-support index, $\bbeta^{b+1} \dord \mu$, and 
the union of the support intervals of $\zeta^m$ for $1\le m \le \ell$ is $[1,b]$.
Thus, by Theorem \ref{thm:real},
if $\ep=\sum_{m=1}^\ell \zeta^m + \bbeta^{b+1}$, then
$x_\ep < x < x_\ep'$.

Notice that by Corollary \ref{cor:LL-decomposition}, given $x\in U$, the \cf\ $\ep\in \cEsbar - \cE^*$ described
 above is uniquely determined, and
this implies the intervals $(x_\ep, x_\ep')$ are disjoint.
This concludes the proof of the fact that $U$ is a disjoint union of the intervals $(x_\ep, x_\ep')$.

\end{proof}

The length of the interval $(x_\ep, x_{\ep'})$ described in Proposition \ref{prop:measurable}
is 
\begin{equation}
I_b:=\wt\ome^b - \sum\bbeta^{b+1} \wt Q \label{eq:Ib}
\end{equation} 
 where $b\ge 1$, and 
to find the measure of $U$, we need to figure out 
the number of intervals $(x_\mu, x_{\mu'})$ where $\mu\in\cEsbar - \cE^*$, for which
$x_{\mu'}-x_\mu = I_b$.
Recall the list $L=(e_1,\dots, e_N)$, which determines the periodic \zec\ collection $\cE^*$.
\begin{lemma}\label{lem:count-proper}
Let $i\le b$ be two positive integers such that $b-i+1\equiv j \moD N$ for some $1\le j\le N$.
The number of proper $\Lstar$-blocks with \supint\ $[i,b]$ is equal to $e_j$.
\end{lemma}

\begin{proof} Let $\zeta$ be a member of $\cE^*$ supported on $[i,b]$.
Notice that $\zeta$ is a proper $\Lstar$-block with \supint\ $[i,b]$ if and only if
$0\le \zeta_b <e_j$ and $\zeta_k =\bbeta^i_k$ for all $k$ such that $i\le k<b$.
Since the number of the possibilities for $\zeta_b$ is $e_j$, we prove the lemma.

\end{proof}

The value  $C_b$ described in Proposition \ref{prop:Cn} below
is the number of intervals $(x_\mu, x_{\mu'})$ whose length is $I_b$ defined in 
(\ref{eq:Ib}).
\begin{prop}\label{prop:Cn}
Let $S_n$ for $n\ge 1$ be the subset of $\ep\in\cE^*$ such that $\ep=\sum_{m=1}^\ell\zeta^m$
is an $L^*$-block decomposition and the support interval of $\zeta^\ell$ is $[i,n]$ for some $i\le n$, and
let $S_0:=\set{0}$.
Let $C_n$ for $n\ge 0$
be the number of elements in $S_n$.
Then,
\begin{equation}\label{eq:Cn}
\begin{aligned}
C_k &= e_1 C_{k-1} + \cdots + e_{k } C_{0}\ \text{ for }k=1,\dots, N\\
C_n &= \sum_{n=1}^{N-1 } e_k C_{n-k} + (1+e_N) C_{n-N}\ \text{ for all }n\ge N+1.
\end{aligned}
\end{equation}
\end{prop}

\begin{proof}
Let $\zeta$ be a proper $\Lstar$-block with \supint\ $[1,1]$.
Then, by Lemma \ref{lem:count-proper}, there are $e_1$ such blocks, and hence, $C_1=e_1$.
We shall calculate $C_n$ inductively.

Given an integer $n\ge 1$, let $T_i$ for $1\le i\le n$ be the set of proper $\Lstar$-blocks with support interval $[i,n]$.
Notice that the uniqueness of the proper $\Lstar$-block decomposition implies that $S_n$ is the disjoint union of
$S_k+ T_{k+1}:=\set{\ep + \mu : \ep\in S_k,\ \mu \in T_{k+1}}$ for $k=0,\dots,n-1$ where
the cardinality of $S_k+T_{k+1}$ is equal to $C_k \cdot \#T_{k+1}=C_k e_j$ for $1\le j\le N$
such that $n-k \equiv j \moD N$.
Thus, $C_n = C_{n-1} e_1 + C_{n-2} e_2 + \cdots + C_1 e_{t-1}+ C_0 e_t$ where $1\le t\le N$ and
$ t\equiv n \moD N$, and the coefficients of $C_k$ periodically repeat.
Via induction,
if $n\ge N+1$, then $C_n = e_1 C_{n-1} + \cdots + e_{N} C_{n-N}+ (e_1C_{n-N-1}+\cdots +e_t C_0)
	=e_1 C_{n-1} + \cdots + e_{N-1} C_{n-N+1}+(1+e_N) C_{n-N} $.

\end{proof}

\begin{prop}\label{prop:generating}
Let $f(x)=1-\sum_{k=1}^{N-1} e_k x^k- (1+e_N) x^N$.
Then, the generating function $p(x):=\sum_{k=0}^\infty C_k x^k$
is identical to $(1-x^N)/f(x)$ on an open interval containing $0$.
\end{prop}

\begin{proof}
For a non-negative integer $t\le N$,
we have $x^t p(x) = \sum_{k=0}^\infty C_k x^{k+t}= \sum_{k=t}^\infty C_{k-t} x^{k}
=C_0 x^t + \cdots + C_{N-t} x^N +\sum_{k=N+1}^\infty C_{k-t} x^{k}
=\sum_{k=t}^N C_{k-t} x^k +\sum_{k=N+1}^\infty C_{k-t} x^{k}$.
Below we shall calculate $p(x)-\sum_{t=1}^{N-1} e_tx^t p(x) - (1+e_N)x^N p(x)$, and
by (\ref{eq:Cn}), the sums over $k\ge N+1$ shall be cancelled out;
\begin{align*}
p(x) -& (e_1 x p(x) +\cdots + e_{N-1} x^{N-1} p(x)+(1+e_N)x^N p(x ))\\
& =
\sum_{k=0}^N C_{k } x^k +\sum_{k=N+1}^\infty C_{k } x^{k}
\\
&\HS{2em}
-\sum_{t=1}^{N-1}
\left(\sum_{k=t}^{N } e_tC_{k-t} x^k +\sum_{k=N+1}^\infty e_tC_{k-t} x^{k}\right)
- (1+e_N)C_0 x^N -\sum_{k=N+1}^\infty (1+e_N)C_{k-N} x^{k}
\\
&= \sum_{k=0}^N C_{k } x^k
- \sum_{t=1}^{N-1}
\sum_{k=t}^{N } e_tC_{k-t} x^k - (1+e_N)C_0 x^N \\
&= \sum_{k=0}^N C_{k } x^k
- \sum_{t=1}^{N-1}
\sum_{k=t}^{N -1} e_tC_{k-t} x^k- \sum_{t=1}^{N-1}
e_tC_{N-t} x^N - (1+e_N)C_0 x^N \\
&= \sum_{k=0}^N C_{k } x^k
-
\sum_{k=1}^{N-1} x^k \sum_{t=1}^{k} e_tC_{k-t}
- \sum_{t=1}^{N-1}
e_tC_{N-t} x^N - (1+e_N)C_0 x^N \\
&= \sum_{k=0}^N C_{k } x^k
-
\sum_{k=1}^{N-1} C_k x^k - x^N \sum_{t=1}^{N-1}
e_tC_{N-t} - (1+e_N)C_0 x^N =1 - x^N.
\end{align*}
Thus, $p(x)f(x) = 1-x^N$, and it proves the result.
\end{proof}

\begin{deF}\label{def:rho}
Let $\rho$ denote the real number $\sum\bbeta^1\wt Q$ (where $\bbeta^1$ is the
maximal $\Lsbar$-block at index $1$).
\end{deF}

\begin{lemma}\label{lem:rho}
Let $\ell\ge 1$ be an integer.
Then, $\sum\bbeta^\ell \wt Q= \wt\ome^{\ell-1}\rho$.
\end{lemma}

\begin{proof}
Notice that $\bbeta^\ell=(\bar 0,\overline{e_1,\dots,e_N})$ implies
$$
\sum \bbeta^\ell \wt Q
=\sum_{k=\ell}^\infty \bbeta^\ell_k \omtilde^k=\sum_{k=\ell}^\infty \bbeta^1_{k-\ell+1} \omtilde^k
=\omtilde^{\ell-1}\sum_{k=\ell}^\infty \bbeta^1_{k-\ell+1} \omtilde^{k-\ell+1}
=\rho \omtilde^{\ell-1}.$$
\end{proof}

\begin{theorem}\label{thm:measure}
The open subset $U $ has Lesbegue measure $1$.
\end{theorem}

\begin{proof}
Let $\ep = \sum_{m=1}^\ell \zeta^m +\bbeta^{b+1}$ be an $\Lsbar$-\bdecom\ %
where $b\ge 0$ is the largest $\Lsbar$-\supint\ of $\ep$. Let $\ep^0= \sum_{m=1}^\ell \zeta^m$,
$x_\ep = \sum\ep\wt Q$,
and $x_\ep'= \sum( \sum_{m=1}^\ell \zeta^m +\beta^b)\wt Q$, which is interpreted as $1$ if $b=0$,
i.e., $\ep=\bbeta^1$.
By Proposition \ref{prop:measurable}, $U$ is the disjoint union of
$(x_\ep,x_\ep')$ where $\ep^0$ varies over $\cE^*\fin$, and it is \lesb\ measurable.

Recall the set $S_b$ and the number $C_b$ defined in Proposition \ref{prop:Cn}.
Notice that by Lemma \ref{lem:rho},
the length of $(x_\ep,x_\ep')$
is $ \wt\ome^b - \wt\ome^b \rho=\wt\ome^b(1-\rho)$.
By Proposition \ref{prop:generating},
the measure of $U$ is
$$\sum_{b=0}^\infty C_b \wt \ome^b (1-\rho) = (1-\rho)\frac{1-\wt\ome^N}{f(\wt\ome)}$$
where $f(x)$ is the characteristic polynomial defined in the proposition.
\newcommand{\wto}{\wt\ome}
Let us rearrange $\rho$ as follows.
Notice that $ e_1 \wto+\cdots +e_N \wto^N=1-f(\wto)-\wto^N$:
\begin{align*}
\rho& =( e_1 \wto+\cdots +e_N \wto^N) + \wto^N( e_1 \wto+\cdots +e_N \wto^N)+\cdots\\
&=(1-f(\wto)-\wto^N)(1+\wto^N+\wto^{2N}+\cdots)
=1-\frac{ f(\wto) }{1-\wto^N}\\
&\implies (1-\rho)\frac{1-\wt\ome^N}{f(\wt\ome)}=1.
\end{align*}
Therefore, the measure of $U$ is equal to $1$.
\end{proof}

\subsection{The local and global maximum and minimum values}

Theorem \ref{thm:local-extrm} below is the technical version of Theorem \ref{thm:local-extremum}, and we prove it in this section.
First notice that by Theorem \ref{thm:continuous},
$\delta^*$ is locally decreasing on $U $, and hence, $\delta^*(x)$
is not a local extremum value if $x\in U$.

\begin{theorem}\label{thm:local-extrm}
The function $\delta^*$ does not assume a local extremum value at points in
$\eval_{\wt Q}((\cEtilde^*-\cEsbar)\cup(\cE^*-\cE^*\fin))$.
It assumes a local minimum value at points in $\cE^*\fin$, and a local maximum value
at points in $\eval_{\wt Q}(\cEsbar - \cE^*)$.
\end{theorem}

The following terminology is explained in
Lemma \ref{lem:upper-bound} below.
\begin{deF}\label{def:p-bound}\rm
Let $p$ denote the smallest positive integer such that
$\gamma \omtilde^{p-1}< \ome^p$, and let us call the integer
{\it the exponent of the generic upper bound of $\delta^*$}.
\end{deF}

\begin{lemma}\label{lem:upper-bound}
If $\ep\in\cEtilde^*$, then
$\delta^*[\ep]<\sum \revbar(\ep) Q/\sum \ep \wt Q<\frac1\gamma(\ome/\omtilde)^p$.
\end{lemma}
\begin{proof}
Notice that $ \sum \ep \wt Q<\left( \sum \ep \wt Q\right)^\gamma$ for all  $\ep\in \cEtilde^*$, so
$\delta^*[\ep]<\sum \revbar(\ep) Q/\sum \ep \wt Q$.
Since $\revbar(\ep)\le \bbeta^1$, we have $\sum \revbar(\ep) Q\le 1$.
Since the set of values of $\delta^*$ is equal to that of $\delta_1^*:[\wt\ome,1)\to \real$, we may assume that $\ep_1\ge 1$, $\sum\ep \wt Q \ge \omtilde$, and hence,
$ \delta^*[\ep]<\sum \revbar(\ep) Q/\sum \ep \wt Q \le 1/\omtilde < \frac1\gamma(\ome/\omtilde)^p$.
\end{proof}

\begin{lemma}\label{lem:gamma-ineq}
Let $\ep\in\cEtilde^*$, and $x=\sum \ep \wt Q$, $y=\sum \revbar(\ep) Q$.
If
$\Delta x$ and $\Delta y$ are positive real numbers such that
$\ome^p/\wt\ome^p\le \Delta y /\Delta x$, then
$y/x^\gamma < (y+\Delta y)/(x+ \Delta x)^\gamma$.
\end{lemma}
\begin{proof}
Since $\frac{\ome^p}{\omtilde^p}\le \frac{\Delta y}{\Delta x}$, Lemma \ref{lem:upper-bound} implies
$\frac yx <\frac1\gamma \frac{\ome^p}{\omtilde^p}\le \frac1\gamma \frac{\Delta y}{\Delta x}$.
Notice that $\frac yx < \frac1\gamma \frac{\Delta y}{\Delta x}
\implies
1 + \gamma \frac{\Delta x}x< 1 + \frac{\Delta y }y$.
Since $(1+t)^\gamma < 1+ \gamma t$ for real numbers $t>0$, we have
$\left( 1 + \frac{\Delta x}x\right)^\gamma < 1 + \gamma \frac{\Delta x}x$, and hence,
\begin{gather*}
\left( 1 + \frac{\Delta x}x\right)^\gamma <1 + \gamma \frac{\Delta x}x< 1 + \frac{\Delta y }y
\implies 
\left( 1 + \frac{\Delta x}x\right)^\gamma < 1 + \frac{\Delta y }y
\implies \frac y{x^\gamma} < \frac{ y + \Delta y }{ ( x + \Delta x)^\gamma }
.\notag
\end{gather*}
\end{proof}

The following corollary is useful for further reducing the finite cases listed in Theorem
\ref{thm:main-theorem} and Theorem \ref{thm:main-minimum} below.
\begin{cor}\label{thm:add-one-term}
Suppose that $\ep\in\cEtilde^*$. If $n\ge p$, then
$$ \delta^*[\ep]=\frac{ \sum\revbar(\ep) Q }{ \left(\sum\ep \wt Q \right)^\gamma}
<\frac{ \sum\revbar(\ep) Q + Q_n }{\left(\sum\ep \wt Q + \wt Q_n \right)^\gamma}.$$
\end{cor}

\begin{proof}
If $p\le n$, then $\ome^p/\wt\ome^p \le Q_n/\wt Q_n = \ome^n/\wt \ome^n$.
By Lemma \ref{lem:gamma-ineq}, we prove the result.

\end{proof}

\begin{prop}\label{thm:local-minimum}
The function $\delta^*$ assumes a local minimum value at the points in $\eval_{\wt Q}(\cE^*\fin)$.
\end{prop}

\begin{proof}
Let $\ep=\sum_{m=1}^\ell \zeta^m$ be an $\Lstar$-\bdecom\ %
where $\zeta^\ell$ is nonzero, and $[i,b]$ is the \supint\ of $\zeta^\ell$, and let $x=\sum\ep\wt Q$.
For sufficiently small positive real number $\Delta x$,
we have $x+\Delta x = \sum\ep^+ \wt Q$ where $\ep^+\in \cEtilde^*$,
and by Lemma \ref{lem:ep-converge},
$\ep^+ = \ep + \tau$ where
$\tau\in \cEtilde^*$ such that $\ord^*(\tau)=n>b+1$ is sufficiently large, i.e.,
it is an $\wt L^*$-\decom.
By Lemma \ref{lem:ep-bar-real} and \ref{lem:ep+-}, $\revbar(\ep^+)=\ep + \revbar(\tau)$, we have
$\Delta y := \sum\revbar(\ep^+) Q - \sum\ep Q =\sum\revbar(\tau) Q \ge \ome^n$ and
$\Delta x = \sum \tau \wt Q < \wt\ome^{n-1}$ by Theorem \ref{thm:real}.
Thus, $\Delta y /\Delta x >\wt \ome (\ome/\wt \ome)^n\to \infty $ as $n\to \infty$.
By Lemma \ref{lem:gamma-ineq},
$\delta^*[\ep]<\delta^*[\ep^+]$.

Since $\zeta^\ell$ is non-zero,
there is a largest index $r$ such that $\zeta^\ell_r>0$, and
$\xi:=\zeta^\ell - \beta^r$ is a proper $\Lstar$-block with support interval $[i,r]$.
Then, $\ep^-:=\sum_{m=1}^{\ell-1} \zeta^m + \xi + \bbeta^{r+1}\in \cEsbar$
is an $\Lsbar$-\bdecom,
and let $x^-:=\sum\ep^- \wt Q$.
By Theorem \ref{thm:continuous} and Proposition \ref{prop:measurable}, $\delta^*$ is decreasing on $(x^-,x)$, and
the continuity of $\delta^*$ implies that $\delta^*(x_1)>\delta^*(x)$ for all $x_1\in (x^-,x)$.

\end{proof}

\begin{prop}\label{thm:not-local-extrm}
The function $\delta^*$ does not assume a local extremum value at the points in
$\eval_{\wt Q}(\cE^*-\cE^*\fin)$.
\end{prop}

\begin{proof}
Let $\ep=\sum_{m=1}^\infty \zeta^m $ be the non-zero $\Lstar$-\bdecom.
For each $m$, there is the largest index $r$ in the \supint\ of $\zeta^m$ and
$\zeta^m_r>0$.
Then, $\ep - \beta^r \in \cE^*$, and by Corollary \ref{thm:add-one-term},
$\delta^*[\ep - \beta^r] <\delta^*[\ep]$ for infinitely many $r$.
So, $\delta^*[\ep]$ is not a local minimum.

Let us show it is not a local maximum.
Given $t\ge 1$, let $[i,\ell-1]$ be the support interval of $\zeta^t$.
Let $\ep^+:= \sum_{m=1}^t \zeta^m + \beta^{\ell-1} + \beta^{r}$ where $r=\ell +N $.
Then, the proper $\Lstar$-block $\xi:=\zeta^t + \beta^{\ell-1}$ has
  \supint\ $[i,c]$ such that
$c\le \ell+N-1$ since $\cE^*$ is periodic and $\bbeta^i_{i+kN}=e_1\ge 1$ for $k\ge 0$.
Since $r>c$, $\ep^+=\sum_{m=1}^{t-1}\zeta^m + \xi + \beta^{r}$ is an $\Lstar$-\bdecom. Thus,
\begin{gather*}
\begin{aligned}
\Delta y:&=\sum\ep^+ Q - \sum \ep Q=
Q_{\ell-1} + Q_r -\sum(\tsum_{m=t+1}^\infty \zeta^m) Q>Q_r=\ome^r,\\
\Delta x:&=\sum\ep^+\wt Q - \sum \ep\wt Q
	= \wt Q_{\ell-1} + \wt Q_r -\sum(\tsum_{m=t+1}^\infty \zeta^m)\wt Q \le \wt Q_{\ell-1} + \wt Q_r
	<2 \wt \ome^{\ell -1 }
\end{aligned}\\
	\implies
	\frac{\Delta y}{\Delta x } > \frac{ \ome^r}{2\wt\ome^{\ell -1 }}
	=\frac{ \ome^{\ell+N-1}}{2\wt\ome^{\ell -1 }},\quad
	\text{where $ \ome^\ell /\wt \ome^\ell \to \infty $ as $\ell\to\infty$.}
\end{gather*}
Thus, for sufficiently large $t$,
$\Delta y /\Delta x > \ome^p/\wt \ome^p$.
Notice that $y + \Delta y = \sum\ep^+ Q=\sum \revbar(\ep^+) Q$ and $x + \Delta x = \sum\ep^+ \wt Q$
where $\ep^+\in\cE^*$.
By Lemma \ref{lem:gamma-ineq}, $\delta^*[\ep]<\delta^*[\ep^+]$.
Since $t$ can be arbitrarily large, this proves that  $\delta^*[\ep]$ is not a local maximum.
\end{proof}

\begin{prop}\label{thm:local-max}
The function $\delta^*$ assumes a local maximum value at the points $x$ in $\eval_{\wt Q}(\cEsbar - \cE^*)$.
Moreover, $ x\in\eval_{\wt Q}(\cEsbar - \cE^*)$ is locally the only value that attains the local maximum.
\end{prop}

\begin{proof}
Let $\ep=\sum_{m=1}^\ell \zeta^\ell + \bbeta^{b+1}$ be an $\Lsbar$-\bdecom\ of $\ep\in\cEsbar-\cE^*$,
and let $x=\sum\ep\wt Q$.
By Theorem \ref{thm:continuous} and Proposition \ref{prop:measurable},
we have $\delta^*(x)>\delta^*(x')$ for all values $x'>x$ that are  sufficiently close to $x$.

\newcommand{\xnhat}{\wh x_n}
Let us show that $\delta^*(x)>\delta^*(x')$ for all values $x'<x$ that are  sufficiently close to $x$.
Given an integer $n\ge b+1$, let $\ep^n:= \sum_{m=1}^\ell \zeta^\ell + \res_n(\bbeta^{b+1})\in \cE^*\fin$,
  let $x_n:=\sum\ep^n \wt Q$, which approaches $x$ as $n\to\infty$, and let $\xnhat \in I_n$ such that 
  $\delta^*(\xnhat)=\max\set{ z \in I_n : \delta^*(z)}$.
  Let  $\xnhat=\sum \ep^- \wt Q$  for $\ep^-\in\cEtilde^*$.
  By Proposition \ref{thm:local-minimum},
  $\delta^*(x_n)$ is a local minimum, and hence, $\xnhat\ne x_n$.
  Let us claim that  
$\ep^-\in \cEsbar-\cE^*$.
By Theorem \ref{thm:measure}, Proposition \ref{thm:local-minimum} and \ref{thm:not-local-extrm}, 
if $\ep^-\in\cE^*\cup(\cEtilde^*-\cEsbar)$ and $\xnhat > x_n$, then $\delta^*$ assumes a value higher than $\delta^*(\xnhat)$ at  values arbitrarily  near $\xnhat$, which contradicts the choice of $\xnhat$.

\renewcommand{\xnhat}{x'}
Below we shall show that if
 $x'=\sum \ep' \wt Q < x$, $\ep'\in \cEsbar-\cE^*$, and $x'$ is sufficiently close to $x$,
 then 
  $\delta^*(\xnhat) <
\delta^*(x)$. 
This shall imply that $\delta^*(x)$ is a local maximum, and 
$x$ is locally the only  value that attains the local maximum.
 
Let $x'=\sum \ep' \wt Q < x$ where $\ep'\in \cEsbar-\cE^*$ such that $x'$ is sufficiently close to $x$.
By Lemma \ref{lem:ep-converge} applied to $\cEtilde^*$,
there are proper $\Lstar$-blocks $\zeta^{m}$ for $\ell+1\le m\le t$
such that $\ep'=\sum_{m=1}^\ell \zeta^\ell + \sum_{m=\ell+1}^t \zeta^m + \bbeta^{e+1}$
is an $\Lsbar$-\bdecom, and $[b+1,c]$ is the \supint\ of $\zeta^{\ell+1}$ where $c\to\infty$ as $\xnhat\to x$.
Notice that the periodic property of $\cE^*$ and Lemma \ref{lem:ep-converge} imply that
there is an index $r$ such that $c+1\le r\le c+N$ and
$\res^r(\bbeta^{b+1}) =(\bar 0,\overline{e_1,\dots,e_N})$.
Recall $\rho:=\sum\bbeta^1\wt Q$ and Lemma \ref{lem:rho}. Then,
\begin{align*}
\Delta x &:= x - \xnhat=\sum\ep \wt Q - \sum\ep'\wt Q
	\le \rho\wt \ome^b - \sum \zeta^{\ell+1} \wt Q\\
	&=  \rho\wt \ome^b -
	(\sum \bbeta^{b+1}\wt Q - \sum_{k=c}^{r-1}(\bbeta^{b+1}_k-\zeta^{\ell+1}_k)\wt\ome^k 
	- \sum \bbeta^{r}\wt Q)\\
	&= \rho\wt \ome^b -
	(\rho\wt \ome^b - \sum_{k=c}^{r-1}(\bbeta^{b+1}_k-\zeta^{\ell+1}_k)\wt\ome^k - \rho \wt
	 \ome^{r-1}),\quad \res^{c+1}(\zeta^{\ell+1})=0.\\
	 \intertext{Let $M$ be the maximum value of $e_k$ for $k=1,\dots,N$;}
\Delta x	 &\le \sum_{k=c}^{r-1}(\bbeta^{b+1}-\zeta^{\ell+1})\wt\ome^k + \rho \wt
	 \ome^{r-1}\le\wt\ome^{r-1}\left( \sum_{k=c}^{r-1}M \wt\ome^{k-r+1}+\rho \right),\ r-N\le c\\
	 &\le \wt\ome^{r-1}\left( \sum_{k=r-N}^{r-1}M \wt\ome^{k-r+1}+\rho\right)
	 \le \wt\ome^{r-1}\left( \sum_{k=0}^{N-1}M \wt\phi^k+\rho \right)
	 \le\lambda \wt\ome^{c} .
\end{align*}
where $\lambda= \sum_{k=0}^{N-1}M \wt\phi^k+\rho$.
Let us estimate $\Delta y:= \sum\ep Q - \sum\ep' Q$.  
Then, by (\ref{eq:full-recursion-real}), $\sum(\sum_{m=\ell+2}^t \zeta^m)Q + \sum\bbeta^{e+1} Q
=\sum(\sum_{m=\ell+2}^t \zeta^m)Q + Q_{e}<Q_c $.
So,
\begin{align*}
\Delta y &=\sum\bbeta^{b+1} Q - \sum\left( \sum_{m=\ell+1}^t \zeta^m +\bbeta^{e+1}\right)Q
\ge \ome^b -\sum ( \zeta^{\ell+1} +\beta^c)Q\\
	&=  \ome^b -
	(\sum \bbeta^{b+1}  Q - \sum_{k=c}^{r-1}(\bbeta^{b+1}_k-\zeta^{\ell+1}_k)\wt\ome^k 
	- \sum \bbeta^{r}  Q + Q_c)\\
&=\ome^b - (\ome^b - \sum_{k=c}^{r-1}(\bbeta^{b+1}_k - \zeta^{\ell+1}_k)\ome^k
- \ome^{r-1}+\ome^c ) = \sum_{k=c}^{r-1}(\bbeta^{b+1}_k - \zeta^{\ell+1}_k)\ome^k
+ \ome^{r-1}-\ome^c \\
&
\ge \ome^c+ \ome^{r-1} -\ome^c= \ome^{r-1}\ge \ome^{c+N-1}.
\end{align*}
Thus,
$
\frac{\Delta y}{\Delta x} \ge \left(\frac\ome{\wt \ome}\right)^c \frac{\ome^{N-1}}{\lambda}$,
and since $c\to\infty$ as $\Delta x\to 0$, it follows $\Delta y/\Delta x\to\infty$.
By Lemma \ref{lem:gamma-ineq}, $\delta^*(x-\Delta x) < \delta^*(x)$.
This proves that $\delta^*(\xnhat) < \delta^*(x)$, and it  concludes the proof of the proposition.

\end{proof}

This concludes the proof of Theorem \ref{thm:local-extrm}.
Using the following theorems, we reduce the task of finding the global extremum values to a finite search,
which is the assertion of Theorem \ref{thm:main-introduction}, Part (2).
Let
\begin{equation}
p^*:= p + 1+\frac{N\ln \phi+\ln(1-\rho+\wt\ome^N)}{ \ln\phitilde - \ln \phi }.
\label{eq:ell-bound}
\end{equation}
\begin{theorem} \label{thm:main-theorem}
The maximum value of $\delta^*$ is obtained only at $\ep\in \cEsbar - \cE^*$.
Suppose that $\ep= \sum_{m=1}^t \zeta^m+\bbeta^{\ell}$ is an $\Lsbar$-block decomposition
with the largest $\Lsbar$-support index $\ell-1\ge 0$, and $\delta^*[\ep]$ is the maximum value.
Then, $\ell < \max\set{2,p^*}$.

\end{theorem}

\begin{proof}
By Theorem \ref{thm:local-extrm}, the global maximum is obtained only on $\cEsbar - \cE^*$,
and let $\ep$ be a member of $\cEsbar - \cE^*$ with the $\Lsbar$-block decomposition as described in the statement
where $\ell \ge 1$.
Suppose that $\ell\ge \max\set{2, p^*}\ge 2$.
Let $\ep^+:= \sum_{m=1}^t \zeta^m + \beta^{\ell-1} + \beta^{r}$ where $r=\ell +N-1$.
Then, $\ep^+\in\cE^*$, and $\res^\ell(\ep)=\bbeta^\ell$.
Thus,
$\Delta y:= \sum\ep^+ Q - \sum \ep Q=Q_r$ by (\ref{eq:full-recursion-real}), and
$\Delta x:=\sum\ep^+ \wt Q - \sum \ep \wt Q
= \wt Q_{\ell-1} -\sum\bbeta^\ell \wt Q + \wt Q_r $.
Thus, by Lemma \ref{lem:rho}, $\Delta x = \wt\ome^{\ell-1} -\rho \omtilde^{\ell-1} + \wt \ome^{\ell+N-1}
=\wt\ome^{\ell-1}(1 -\rho+\wt\ome^N)$, and $\ell\ge p^*$ implies
$$
\frac{\Delta y}{\Delta x } =\frac{\ome^{\ell+N-1}}{\wt\ome^{\ell-1}(1 -\rho+\wt\ome^N)} =
\frac{\wt\phi^{\ell -1}}{\phi^{\ell+N-1}(1 -\rho+\wt\ome^N)}
\ge \frac{\wt\phi^p}{\phi^p}=\frac{\ome^p}{\wt \ome^p}.
$$
By Lemma \ref{lem:gamma-ineq}, $\delta^*[\ep]<\delta^*[\ep^+]$,
and it implies that $\delta^*[\ep]$ is not a maximum value.
Therefore, $\ell < \max\set{2, p^*}$.

\end{proof}

\begin{theorem}\label{thm:main-minimum}
Let $p^\dagger:=\max\set{2,p}$.
Then, there is $\ep\in\cE^*$ such that $\res^{p^\dagger}(\ep)=0$ and $\delta^*[\ep]$ is the minimum value.
In particular, if $p\le 2$, then the minimum value is
$\dstar[ \beta^1] =1 $.
\end{theorem}
\begin{proof}
By Theorem \ref{thm:local-extrm},
the minimum value is obtained by $\ep\in\cE^*\fin$, i.e.,
$\ep = \sum_{m=1}^\ell \zeta^m$, which is the non-zero $\Lstar$-\bdecom.
Since $\delta_1: [\wt\ome,1) \to \real$ takes the same set of values as $\delta$,
WLOG, assume that $\ep_1\ge 1$.
Let $[i ,s ]$ be the support interval of $\zeta^\ell$, and let $r $ be the largest integer such that
$\zeta^\ell_{r }>0$.
Then, $\zeta^\ell - \beta^{r }$ is a proper $\Lstar$-block with support interval $[i,r]$, and
hence $\ep^-:=\ep-\beta^{r } \in \cE^*\fin$.
Suppose that $\delta^*[\ep]$ is the minimum value of $\delta^*$.
If $r \ge p^\dagger\ge 2$, then $\ep^-_1\ge 1$, so $\ep^-$ is non-zero.
By Lemma \ref{thm:add-one-term}, $\delta^*[\ep^-]<\delta^*[\ep^-+\beta^{r }]=\delta^*[\ep]$ contradicting that $\delta^*[\ep]$ is a minimum.
Thus, $r <p^\dagger$, which proves that $\res^{p^\dagger}(\ep) = 0$.
\end{proof}

\section{Examples}\label{sec:examples}

We consider several examples in this section, and use the results in Section \ref{sec:proofs} to find the global maximum and minimum values of $\delta^*$.
Recall the constant $\al$ defined in Theorem \ref{thm:ratios},
and for the closure of calculations, we introduce the following theorem.
The idea of the proof of the following is also available in \cite[Theorem 2.4]{miller-2021}.
\begin{theorem}\label{thm:alpha}
Let $\cE$ be the periodic \zec\ collections for positive integers determined by a list $L=(e_1,\dots,e_N)$.
Let $f$ be the
characteristic polynomial of   $L $ for positive integers, defined in Definition \ref{def:char-poly}, and let $\phi$ be its dominating (real) zero.
Let $H$ be the fundamental sequence of $\cE$.
Then,
\begin{gather*}
\lim_{n\to\infty} \frac{ H_n}{\phi^{n-1}} =
\frac1{f'(\phi)}
\sum_{k=1}^N \frac{H_k}{(k-1)!} \left[ \frac{d^{k-1}}{dx^{k-1}} \frac{f(x)}{x-\phi}
\right]_{x=0}.
\end{gather*}
If $H_k=B^{k-1}$ for $1\le k\le N$ and $B\ne \phi$, then
$$
\lim_{n\to\infty} \frac{ H_n}{\phi^{n-1}}
= \frac{ f(B) }{(B-\phi)f'(\phi)}.$$

\end{theorem}

\begin{proof}\newcommand{\vecr}{\mathbf{r}}
Suppose that $f$ has no repeated zeros.
Let $\lambda_1,\dots,\lambda_{N-1}$ be the remaining distinct zeros of $f$.
Given a complex number $x$, let $\vecr_k(x) = (\lambda_1^{k-1},\dots,\lambda_{N-1}^{k-1},x)$ for $1\le k\le N$ be row vectors, and
let $M(x_1,\dots,x_N)$ be the matrix whose $k$th row is $\vecr_k(x_k)$
where $x_k$ are complex numbers.
By Binet's formula, $H_n = \sum_{k=1}^{N-1} \al_k \lambda_k^{n-1}
+ \al \phi^{n-1}$ for constants $\al_k$ and $\al$.
By specializing this formula at $1\le n\le N$, we have a system of linear equations for $\al_k$ and $\al$, and by Cramer's rule, we have
$\al = \det M(H_1,\dots,H_N)/ \det M(1,\phi,\dots,\phi^{N-1})$.

By the determinant formula of the Vandermonde matrix,
$\det M(1,\phi,\dots,\phi^{N-1}) =
E\prod_{k=1}^{N-1} (\phi - \lambda_k)$ where $E$ is a quantity
determined only by $\lambda_1,\dots,\lambda_{N-1}$, and in particular, $E$ is independent
of $\phi$.
Since $f(x) = (x-\phi)g(x)$ for $g(x)\in \real[x]$, $f'(\phi)=g(\phi)
=\prod_{k=1}^{N-1} (\phi - \lambda_k)$.
Hence, $\det M(1,\phi,\dots,\phi^{N-1}) =E f'(\phi)$.

Let $x$ be a polynomial variable.
Then, $\det M(1,x,\dots,x^{N-1})=Ef (x)/(x-\phi)$ since it is a Vandermonde matrix.
For $1\le k \le N$, let $C_k$ be the cofactor of $M(1,x,\dots,x^{N-1})$ at the position
$(k,N)$, so that $\det M(1,x,\dots,x^{N-1})=
\sum_{k=0}^{N-1} C_{k+1} x^k$. Then,
for $1\le k\le N$,
\begin{gather*}
Ef (x)/(x-\phi) =\sum_{k=0}^{N-1} C_{k+1} x^k \implies
C_k = \frac E{(k-1)!} \left[ \frac{d^{k-1}}{dx^{k-1}} \frac{f(x)}{x-\phi}
\right]_{x=0}\\
\implies
\det M(H_1,\dots,H_N)= \sum_{k=1}^{N } C_k H_k
=\sum_{k=1}^{N } \frac{EH_k}{(k-1)!} \left[ \frac{d^{k-1}}{dx^{k-1}} \frac{f(x)}{x-\phi}
\right]_{x=0}\\
\implies
\al =\frac{\det M(H_1,\dots,H_N)}{ \det M(1,\phi,\dots,\phi^{N-1})}= \frac1{f'(\phi)}
\sum_{k=1}^N \frac{H_k}{(k-1)!} \left[ \frac{d^{k-1}}{dx^{k-1}} \frac{f(x)}{x-\phi}
\right]_{x=0}.\\
\intertext{If $H_k=B^{k-1}$ for $1\le k\le N$ and $B\ne \phi$, then
}
\det M(H_1,\dots,H_N)=E \prod_{k=1}^{N-1} (B - \lambda_k)
= E \frac{ f(B) }{B-\phi}
\implies \al =
\frac{\det M(H_1,\dots,H_N)}{ \det M(1,\phi,\dots,\phi^{N-1})}
= \frac{ f(B) }{(B-\phi)f'(\phi)}.
\end{gather*}

Suppose that $f$ has repeated zeros. The dominant zero $\phi$ remains simple,
and let $d_k$ be the multiplicity of $\lambda_k$ for $1\le k \le N'$ where $N'$ is the number of distinct zeros of $f$ other than $\phi$.
Let
the rows of $M(x_1,\dots,x_N)$ consist of
the $d_k-1$ consecutive derivatives of $\lambda_j^{k-1}$ where
$\lambda_j$ is treated as a variable.
For example, if $\lambda_1$ is the only repeated zero, and it has multiplicity $3$,
then
\begin{align*}
\vecr_k (x_k)& =(\lambda_1^{k-1},(k-1)\lambda_1^{k-2},(k-1)(k-2) \lambda_1^{k-3},\lambda_2^{k-1},\dots,\lambda_{N-3}^{k-1},x_k)\\
&=\left(\lambda_1^{k-1},\tfrac{k-1}{\lambda_1}\lambda_1^{k-1},
\tfrac{(k-1)(k-2)}{\lambda_1^2} \lambda_1^{k-1},\lambda_3^{k-1},\dots,\lambda_{N-1}^{k-1},x_k\right).
\end{align*}
See \cite[Section 1]{chang-PME} for more examples.
Let us use these components as a basis for Binet's formula.
By \cite[Theorem 1]{chang-PME} or \cite{Kalman},
$\det M(1,x,\dots,x^{N-1})= E f(x)/(x-\phi)$ where $E$ depends only on $\lambda_k$ and their multiplicities, and in particular, $E$ is independent of $x$.
Also, $\det M(1,\phi,\dots,\phi^{N-1})= E f'( \phi)$.
Using this result, the argument for $f$ with no repeated roots applies to this situation with no difficulty, and we leave it to the reader.

\end{proof}

Recall the exponent $p$ of the generic upper bound of $\delta^*$ from Definition \ref{def:p-bound}, and the values $p^*$ and $p^\dagger$
described in Theorem \ref{thm:main-theorem} and \ref{thm:main-minimum}.

\subsection{The $N$th order \zec\ base-$N$ expansions}\label{sec:base-N}
Let $\cE^*$ be the periodic collection for $\OI$ determined by $L=(1,0)$ and let $\cEtilde$ be
the one for $\OI$ determined by the list $\wt L=(1,1)$.
This is an example introduced in Section \ref{sec:introduction} called the \zec\ binary expansions.
Then, the exponent of generic upper bound of $\delta^*$ is $p=2$, and $p^*\approx 4.998$.
The value of $p^\dagger$ is $2$, so $\dstar[\beta^1]=1$ is the minimum value.
For the maximum values, consider $\ep=\sum_{m=1}^t \zeta^m + \bbeta^{\ell}$,
which is an $\Lsbar$-block decomposition with the largest $\Lsbar$-support index $\ell-1$.
By Theorem \ref{thm:main-theorem} with $p=2$,
we have $\ell\le 4$, and hence, the only possibilities of $\ep$ are $ (1,0,0, \overline{1,0})$ and $
\bbeta^1$.
It turns out that $\delta^*[\bbeta^1]$ is the largest, and its value is $1/\rho^\gamma=(3/2)^\gamma$.
Using Theorem \ref{thm:alpha} and Proposition \ref{prop:del-delstar}, we obtain the result (\ref{eq:lim-inf-sup-2}).

We may generalize it to the setting of the following two lists of length $N$.
\begin{deF}
\rm \label{def:N-Zec}
Let $N\ge 2$ be an integer.
The periodic \zec\ collections $\cE$ for positive integers and $\cE^*$
determined by $L=(1,1,\dots,1,0)$ where 
$1$ is repeated $N-1$ times are called
{\it the $N$th order \zec\ collections for positive integers and for $\OI$}, respectively. 
The periodic \zec\ collections $\cEtilde$ for positive integers and $\cEtilde^*$ for $\OI$ determined by
$\wt L=(N-1,N-1,\dots,N-1)$ are called  {\it the base-$N$ collections for positive integers and 
for $\OI$}, respectively.
\end{deF}

Let us demonstrate the third order \zec\ collections for positive integers, i.e.,   $N=3$.
Since $L=(1,1,0)$, the collection $\cE$ consists of \cf s that have no three consecutive $1$s.
For example, $100= 3^0+ 2\cdot 3^2+3^4 =\wt H_1 +2 \wt H_3 + \wt H_5$ is not a $3$th order \zec\ expression
while $1000 = \wt H_1 + \wt H_4 + \wt H_6 + \wt H_{7}$ is a $3$th order \zec\ expression.

Theorem \ref{thm:main-theorem} is what we could do for general cases, but
for the setup considered in Definition \ref{def:N-Zec}, we obtain
a better version of the theorem using Lemma \ref{lem:gamma-ineq} alone.
\begin{lemma}\label{lem:shift}
Let $\cE^*$ be the periodic \zec\ collection defined in Definition \ref{def:N-Zec} where $N\ge 3$, and let $\ep\in\cE^*$.
Then, the exponent of generic upper bound of $\delta^*$ is $p=1$.
If $\ep_n=1$ and $\zeta$ is a proper $\Lstar$-block of $\ep$ with \supint\ $[i,n-1]$ for $n\ge 3$,
then $\ep':=\ep +\beta^{n-1} - \beta^n\in\cE^*$ and $\delta^*[\ep] < \delta^*[\ep']$.

\end{lemma}

\begin{proof}
Let us prove that $p=1$.
Recall the notations from Notation \ref{def:bbeta}.
Notice that $\wt\ome = 1/N$, and the characteristic polynomial of $L$ for positive integers is $f_N(x)=x^N -(x^{N-1}+\cdots +x+1)$.
Since $f_N(2) = 2^N- (2^N-1)=1$, we have $\phi<2$, and hence, $\ome>1/2$.
If $\phi'$ is the dominant real zero of $f_{N-1}$, then
the collection for positive integers determined by $L_0=(1,\dots,1,0)$ with $N-2$ copies of $1$ is a subcollection of
the one determined by $L$, and hence, by Theorem \ref{thm:ratios}, we have $\phi'<\phi$.
This implies that $\ome\le \ome_0$ where $\ome_0$ is the reciprocal of the dominant zero of $x^3 -(x^2+x+1)$, which is $< .55$.
If $N\ge 4$ and $p=1$,
then
\begin{gather*}
\phi^{p }\ln \phi<
2 \ln 2 =\ln 4 \le N^{p-1} \ln N \implies
\gamma \wt\ome^{p-1} < \ome^p 
\end{gather*}
where $\gamma$ is the ratio defined in Theorem \ref{thm:ratios}.
The inequality also holds for $N=3$ as well.

Suppose that 
$\ep_n=1$, and that $\zeta$ is a proper $\Lstar$-block of $\ep$ with \supint\ $[i,n-1]$ for $n\ge 3$.
Let  $\ep':=\ep +\beta^{n-1} - \beta^n\in\cE^*$, and recall the fundamental sequences 
$Q$ and $\wt Q$ from  Notation \ref{def:bbeta}
Let $x=\sum\ep\wt Q$, $y=\sum\ep Q$, $x'=\sum\ep'\wt Q$, and $y'=\sum\ep' Q$.
Let us show that $\ep'\in\cE^*$.
If $(\zeta_{n-2},\zeta_{n-1})=(0,0)$,
then
$(\ep'_{n-2},\ep'_{n-1},\ep'_n)=(0,1,0)$, and hence, $\ep'\in \cE^*$.
If $(\zeta_{n-2},\zeta_{n-1})=(1,0)$,
then by the definition of the \supint\ of a proper $\Lstar$-block,
the number of $1$s preceding the entry $\zeta_{n-1}=0$ must be $<N-1$
since $\ep_n=1$.
Thus, $(\ep'_{n-2},\ep'_{n-1},\ep'_n)=(1,1,0)$, and
the number of $1$s preceding the entry $\ep'_n=0$ is less than or equal to $N-1$.
Hence, $\ep'\in\cE^*$.

Thus, we have $\delta^*(x') = y'/(x')^\gamma$.
Let $\Delta x = x'-x=\wt\ome^{n-1} - \wt\ome^n$
and $\Delta y = y' -y =\ome^{n-1} - \ome^n=\ome^{n-1}(1-\ome)$.
Notice that $n\ge 3$, and  $(1-x)x$ is an increasing function on $[0,1/2]$ and
decreasing on $[1/2,1]$.
Then, $\frac1N=\wt\ome\le \frac13< \frac12<\ome\le \ome_0<.55$, and
$ \frac12 - \wt\ome \ge \frac16>0.1$ and $\ome - \frac12 <.05$.
Since the graph of $(1-x)x$ is symmetric about the vertical line $x=\frac12$,
it follows that $(1-\wt\ome)\wt\ome < (1-\ome)\ome$, and hence,
\begin{gather}
\frac{\Delta y}{\Delta x}=\frac{1-\ome}{1-\wt\ome}\left(\frac{\ome^{n-1}}{\wt\ome^{n-1}}\right)\ge
\frac{1-\ome}{1-\wt\ome}\left(\frac{\ome^2}{\wt\ome^2}\right)>
\frac{\ome }{\wt\ome }=
\frac{\ome^p }{\wt\ome^p }.\label{eq:shift}
\end{gather}
Thus, by Lemma \ref{lem:gamma-ineq}, $\delta^*(x)<\delta^*(x+\Delta x)$ where $N\ge 3$.

\end{proof}

By Theorem \ref{thm:local-max}, if $\delta^*[\ep]$
is the maximum and $\ep_1\ge 1$, then $\ep\in \cEsbar - \cE^*$.
Let $\ep=\sum_{m=1}^\ell\zeta^m + \bbeta^{b+1}$ be an $\Lsbar$-\bdecom\ where $b\ge 2$.
Since $\zeta^\ell$ is a proper $\Lstar$-block, we may shift $\ep_{b+1}=1$, and obtain a higher value of $\delta^*$ by Lemma \ref{lem:shift}.
Thus, $b<2$, and $\ep_1\ge 1$ implies that $b=0$, i.e., $\ep = \bbeta^1$.
By Theorem \ref{thm:main-minimum}, the minimum is $\delta^*[\beta^1]$.
By Theorem \ref{thm:alpha}, Proposition \ref{prop:del-delstar}, and the above calculations, we proved the following.
\begin{theorem}\label{thm:N-order}
Let $N\ge 2$ be an integer, let $\phi$ be the dominant
positive real zero of $f(x)=x^N - \sum_{k=0}^{N-1} x^k$,
and let $\gamma:=\log_N \phi$.
Let $z(x)$ be the number of non-negative integers $n< x$ whose base-$N$ expansions  
 are  $N$th order \zec\ expressions.
Then,
\begin{gather}
\lim\sup_x\ \frac{z(x)}{x^\gamma}
	= \frac{\al}{P^\gamma\phi^N},\quad
\lim\inf_x\ \frac{z(x)}{x^\gamma} =\al
\label{eq:lim-inf-sup-N}	\\
\intertext{where}
\al:= \frac{f(N)}{(N-\phi)f'(\phi)},\quad
P:=\frac{N^{N-1}-1}{(N-1)(N^N-1)N^{N-1}}.	\notag
\end{gather}
\end{theorem}

\subsection{Non-base-$N$ expansions}\label{sec:non-base}

Let $\cE$ be the collection for positive integers determined by $L=(2,0,1)$ and let $\cEtilde$ be the one determined by $\wt L=(10,4)$.
By (\ref{eq:full-recursion}), the \funds\ $H$ of $\cE$ is given by
$( H_1, H_2, H_3 )=(1,3,7)$ and
$
H_n = 2 H_{n-1} + 2H_{n-3} $ for $n\ge 4$, and the \funds\ $\wt H$ of $\cEtilde$ is given by
$(\wt H_1,\wt H_2 )=(1,11)$ and
$
\wt H_n = 10\wt H_{n-1} +5 \wt H_{n-2} $ for $n\ge 3$.
The characteristic polynomials of $L$ and $\wt L$ for positive integers are $f(x) = x^3-2x^2-2$ and $\wt f(x) = x^2 - 10x - 5$,
respectively, and $\phi\approx 2.36$ and $\wt \phi\approx 10.48$ are their dominant (real) zeros, respectively.
Then, the exponent of generic upper bound of $\delta^*$ is $p=1$, and $p^*\approx 3.59$.

By Theorem \ref{thm:main-theorem} and Corollary \ref{thm:add-one-term},
the following are the
possibilities of $\ep\in \cEsbar -\cE^*$ for the maximum values of $\delta_1$.
The symbol $\star$ indicates the repeating blocks $(\overline{2,0,1})$:
\begin{gather*}
\bbeta^1,\ (1,\star),\ (1,1,\star).
\end{gather*}
For example, $\ep=(1, 0, \star)\in \cEsbar - \cE^*$ is not listed above since 	
by Corollary \ref{thm:add-one-term}, $\delta^*[\ep]<\delta^*[\ep']$ where
$\ep'=(1, 1, \star)\in\cEsbar - \cE^*$.
According to the numerical calculation of $\delta^*$
at the above three values, $\delta^*$ attains the highest value at $\ep^M:=(1,\overline{2,0,1})$,
and
$$\delta^*[\ep^M]=\frac{\ome + \sum\bbeta^2 Q}{(\wt \ome + \sum \bbeta^2 \wt Q)^\gamma}
= \frac{ 2 \ome }{ (\wt \ome + \wt \ome \rho)^\gamma}
= \frac{ 2 }{ (1+\rho)^\gamma}\approx 1.8757.$$
Since $p=1$, the minimum of $\delta^*$ is $\delta^*[\beta^1]
=1$.

Using Theorem \ref{thm:ratios} and \ref{thm:alpha},
we find the exact values of the maximum and minimum values to be as follows:
\begin{gather*}
\al:=\lim_{n\to\infty} \frac{H_n}{\phi^{n-1}}=\frac{4+3\phi+17\phi^2}{86},\quad
\wt \al:=\lim_{n\to\infty} \frac{\wt H_n}{\wt\phi^{n-1}}=\frac{\phitilde}{10}\\
\rho=(2\wt \ome + \wt \ome^3) + \wt \ome^3(2\wt \ome + \wt \ome^3) +\cdots
=\frac{2\wt \phi^2 +1}{\wt \phi^3-1}=\frac{-227 + 25\wt \phi}{182}.
\end{gather*}
Thus, by Proposition \ref{prop:del-delstar}, we proved
\begin{align*}
\lim\sup_x\ \frac{z(x)}{x^\gamma}
	&=\frac{2\al}{(\wt \al (1+\rho))^\gamma}
	=\frac{ 364^\gamma}{43}\frac{4+3\phi+17\phi^2}{(25+41\phitilde)^\gamma}
	\approx 2.2666 ,\\
\lim\inf_x\ \frac{z(x)}{x^\gamma}
&=\frac{\al}{\wt \al^\gamma}=\frac{10^\gamma}{86}
		(3+13\phi+2\phi^2)
\approx 1.2084.
\end{align*}

\section{Future work}\label{sec:gen-reg-seq}

From Definition \ref{def:N-Zec}, recall  base-$N$ collections for positive integers.
If  $\cEtilde$  is a base-$N$ collection for positive integers, and $\cE$ is a periodic \zec\ 
subcollection of $\cEtilde$ for positive integers,
then the counting function $z(x)$ defined in Definition \ref{def:counting-function} turns out to be the summatory function 
of an
$N$-regular sequence $\chi$, which is defined in \cite{shallit}.
In \cite{heuberger}, it is proved that given a regular sequence $\chi$,  the following 
asymptotic relation holds:
\begin{equation}\label{eq:full-asymptotic}
\sum_{k=1}^n \chi(k) \sim n^\gamma \Phi(\set{\log_{\wt \phi}(n)}) 
\end{equation} 
where $\gamma<1$ is a real number, and $\Phi$ is a continuous function on $(-\infty,\infty)$.

If $\cEtilde$ is a periodic \zec\ collection but not a base-$N$ collection,
then the result of \cite{heuberger}
does not apply as the definition of regular sequences
is given in terms of base-$N$ expansions.
In our future work, 
we aim to establish the  definition of {\it generalized regular sequences} associated with a periodic \zec\ collection 
for positive integers such that 
given a generalized regular sequence $\chi$,
 the relation (\ref{eq:full-asymptotic}) holds
 where $\wt \phi$ is the dominant zero defined in Definition \ref{def:char-poly}.

\end{document}